\documentclass[12pt]{article}
\usepackage{url}
\usepackage{hyperref}
\usepackage[final]{optional}
\usepackage[a4 paper, margin=2cm]{geometry}

\usepackage{enumerate}
\usepackage{graphics}
\usepackage{tikz}
\usepackage{tikz-3dplot} 
\usepackage{verbatim}
\usepackage{slashed}
\usepackage{hyperref}

\usepackage{datetime}
\usepackage{color}
\usepackage{amsmath,amsfonts,amsthm,amssymb}
\usepackage{epstopdf}
\usepackage[all,cmtip]{xy}
\usepackage{accents}
\newtheorem{theorem}{Theorem}[section]
\newtheorem{prop}[theorem]{Proposition}
\newtheorem{definition}[theorem]{Definition}
\newtheorem{corollary}[theorem]{Corollary}
\newtheorem{lemma}[theorem]{Lemma}
\newtheorem{remark}[theorem]{Remark}

\newtheorem{ass}{Assumption}
\numberwithin{equation}{section} 

\newcommand{\R}{\mathbb{R}}
\newcommand{\N}{\mathbb{N}}

\newcommand{\Sp}{\mathbb{S}}

\newcommand{\mbb}{\mathbb}
\newcommand{\C}{\mbb{C}}

\newcommand{\lan}{\langle}
\newcommand{\la}{\lambda}
\newcommand{\ra}{\rangle}

\newcommand{\h}{\mathcal{H}}

\newcommand{\In}{\subset}

\newcommand{\om}{\omega}
\newcommand{\dl}{{\delta}}
\newcommand{\Dl}{{\Delta}}

\newcommand{\al}{{\alpha}}

\newcommand{\ed}{{\rm d}}
\newcommand{\id}{\,\,\ed}

\newcommand{\D}{{\nabla}}
\newcommand{\Db}{{\nabla^{\bot}}}
\newcommand{\ti}[1]{{\tilde{#1}}}
\newcommand{\eps}{{\varepsilon}}
\newcommand{\fr}[2]{\frac{#1}{#2}}
\newcommand{\sm}{{\setminus}}

\newcommand{\p}{\varrho}
\newcommand{\pl}[2]{{\frac{\partial #1}{\partial #2}}}
\newcommand{\ppl}[3]{{\frac{\partial^2 #1}{\partial #2 \partial #3}}}

\newcommand{\zb}{\overline{z}}
\newcommand{\db}{\overline{\partial}}
\newcommand{\de}{\partial}
\newcommand{\nn}{\nonumber}
\newcommand{\spt}{{\rm{spt}}}

\newcommand{\Si}{\Sigma}

\newcommand{\twopartdef}[4]
{
	\left\{
		\begin{array}{ll}
			#1 & \mbox{if  } #2 \bigskip \\
			#3 & \mbox{if  } #4
		\end{array}
	\right.
}

\def\Xint#1{\mathchoice
   {\XXint\displaystyle\textstyle{#1}}%
   {\XXint\textstyle\scriptstyle{#1}}%
   {\XXint\scriptstyle\scriptscriptstyle{#1}}%
   {\XXint\scriptscriptstyle\scriptscriptstyle{#1}}%
   \!\int}
\def\XXint#1#2#3{{\setbox0=\hbox{$#1{#2#3}{\int}$}
     \vcenter{\hbox{$#2#3$}}\kern-.5\wd0}}

\def\dashint{\Xint-}

\DeclareMathOperator{\Vol}{\rm Vol}
\DeclareMathOperator{\dist}{\rm dist}

\newcommand{\vlinesub}[1]{\vline_{_{_{_{_{_{_{_{\,#1}}}}}}}}}

\newcommand{\plaz}{\partial_\la z}

\newcommand{\norm}[1]{\Vert#1\Vert}  
 
\newcommand{\na}{\nabla}
\newcommand{\II}{\mathcal{I}}
\newcommand{\VV}{\mathcal{V}}
\newcommand{\lam}{\lambda}
\newcommand{\beq}{\begin{equation}}
\newcommand{\eeq}{\end{equation}}
\newcommand{\beqs}{\begin{equation*}}
\newcommand{\eeqs}{\end{equation*}}
\newcommand{\abs}[1]{\vert#1\vert}
\newcommand{\babs}[1]{\left| #1 \right|} 
\newcommand{\Loj}{\L{}ojasiewicz}
\newcommand{\beqa}{\begin{equation}\begin{aligned}}
\newcommand{\eeqa}{\end{aligned}\end{equation}}
\newcommand{\beqas}{\begin{equation*}\begin{aligned}}
\newcommand{\eeqas}{\end{aligned}\end{equation*}}
\newcommand{\half}{\frac12}
\newcommand{\thalf}{\tfrac12}
\newcommand{\si}{\sigma}
\newcommand{\nuu}{\nu}
\newcommand{\be}{\beta}
\newcommand{\FF}{\mathcal{F}}
\newcommand{\bl}{\pi}
\newcommand{\ZZ}{\mathcal{Z}}
\newcommand{\ZZO}{{\mathcal{Z}_{\lambda_0}}}
\newcommand{\HH}{\mathcal{H}}
\newcommand{\BB}{\mathcal{B}}
\newcommand{\DD}{\mathbb{D}}

\allowdisplaybreaks

\begin{document}
\title{\L{}ojasiewicz inequalities near simple bubble trees}
\author{Andrea Malchiodi, Melanie Rupflin and Ben Sharp}
\date{\today}
\maketitle

\begin{abstract}
	In this paper we prove a gap phenomenon for critical points of the $H$-functional on closed non-spherical surfaces when $H$ is constant, and in this setting furthermore prove that sequences of almost critical points satisfy \L{}ojasiewicz inequalities as they approach the first non-trivial bubble tree.

	To prove these results we derive sufficient conditions for \L{}ojasiewicz inequalities to hold near a finite-dimensional submanifold of almost-critical points for suitable functionals on a Hilbert space. 	
	\end{abstract}

\section{Introduction}

In the study of (almost-)critical points of an energy functional one is often confronted with the problem that the weakly-obtained limiting object does not have the same topology; for example sequences of (almost-)harmonic maps from a surface $\Sigma$ can converge to a whole bubble-tree of harmonic maps rather than just one harmonic map defined on $\Sigma$, blow-ups of singularities of compact solutions of mean curvature flow can be non-compact shrinkers etc. A major challenge in such situations is that the seminal results of Simon \cite{Simon} on \Loj-Simon inequalities, one of the most powerful tools in the analysis of (almost-)critical points of analytic energies, are not applicable.

In spite of the significant obstacle posed by a change in topology, for a few very important problems \Loj-Simon inequalities have been obtained: namely in the major works of Topping \cite{Top04,Top04A} for almost-harmonic maps between two-dimensional spheres, of Colding-Minicozzi \cite{ColdMini15} for generic singularities of the mean curvature flow and of Chodosh-Schulze \cite{ChSchulze} for conical singularities of the mean curvature flow.

In this paper we establish \Loj-Simon inequalities for sequences of maps on non-spherical closed surfaces which are almost-critical points of the $H$-surface energy and weakly converge to a constant map but blow a single bubble: i.e. the maps bubble-converge to the first non-trivial bubble tree.

In fact we develop a general method that allows us to prove \Loj-estimates near a suitable submanifold $\ZZ$ of almost critical points of a functional $\mathcal{I}$ on a Hilbert space $\h$, on which the second variation is non-degenerate in orthogonal directions, as one might obtain by carefully gluing non-degenerate critical points into domains of different topology.

 To this end we will prove an abstract result that gives sufficient conditions 
 on the first and second variation of $\II$ along the finite dimensional submanifold  $\ZZ$, in order for \Loj-estimates to hold in a uniform $\h$-neighbourhood of $\ZZ$, see Theorem \ref{thm:abstract-main}.

We then apply these results for (almost-)critical points of the $H$-functional (for $H\equiv 1$)
\begin{equation}\label{eq:defE}
E(u):= \frac12\int_{\Sigma} |\D u|^2 \id V_g - 2V(u),\quad u\in H^1(\Sigma, \R^3) 
\end{equation} on closed oriented Riemannian surfaces $(\Sigma,g)$ where $V(u)$ is the enclosed volume given by 
$$V(u)=\frac13 \int_\Sigma u\cdot (u_{x_1}\wedge u_{x_2}) \id x $$ in oriented  coordinates $\{x^i\}$, see \eqref{eq:defEcf} for a coordinate-free definition of $E$. We note that Wente's famous inequality \cite{We69} implies that $E(u)$ is well-defined for all $u\in H^1$ (not necessarily bounded), see Remark \ref{rmk:int}. Since $E$ is invariant under translation in the target, we consider this functional simply on the Hilbert space 
\begin{equation*}
\dot H^1(\Sigma,\R^3) := \left\{v\in H^1(\Sigma,\R^3) : \int_\Sigma v \id V_g = 0\right\}	
\end{equation*}
with the inner product $\langle a,b \rangle_{\dot H^1} = \int_\Sigma \D a \cdot \D b \id V_g$. As the functional and this inner product are conformally invariant, we may assume in the following that $g$ is a metric of constant curvature $K_g\in\{-1,0,1\}$ with unit area when $K_g\equiv 0$. 

We recall that critical points of \eqref{eq:defE} solve 
\begin{equation}\label{eq:Heq}
-\Dl_x u =
2 u_{x_1}\wedge u_{x_2},  
\end{equation}
in oriented isothermal coordinates and are thus smooth by Wente's regularity theory \cite{We69}. It is well-known that conformal critical points $u$ to \eqref{eq:Heq} are (branched) immersed constant mean curvature surfaces in $\R^3$. 

There are many notable works studying properties of $H$-surfaces and related problems on two-dimensional domains, including the results in \cite{BC85,We69} which are crucial to the present paper, and we refer the reader to the books of Struwe \cite{Str88,Str08} for an overview of further results.   

One easily checks that, \emph{amongst critical points} $u$ of $E$, $E(u)=V(u)=\frac16 \int_{\Sigma} |\D u |^2$. So, as pointed out by Ge in \cite{G98}, combining the natural bound on the area  $A(u)=\int_\Sigma |u_{x_1}\wedge u_{x_2}|\id x\leq \frac12 \int_\Sigma |\D u|^2$ with the isoperimetric inequality $36\pi |V(u)|^2 \leq A(u)^3$ (see e.g. \cite{Ra47}) and the associated rigidity statement yields 
\begin{equation*}
E(u)\geq \frac{4\pi}{3}
\end{equation*}
 amongst critical points, with equality if and only if $u:\Sp^2\to \R^3$ is a degree one conformal parametrisation of a unit sphere. Here and in the following all quantities on $\Sigma$ are computed using the metric $g$ unless specified otherwise.  

This already implies that there are no critical points of $E$ from a Riemannian surface $\Sigma$ with genus $\gamma \geq 1$ whose energy is $\frac{4\pi}{3}$. While 
it is possible to approach the energy $\frac{4\pi}{3}$ through critical points on degenerating tori, see \cite{We01}, our first result excludes the 
 possibility that the energy levels of critical points from any \emph{fixed} surface of genus $\gamma \geq 1$ accumulate near this value.

\begin{theorem} \label{t:gap}
	Let $(\Sigma, g)$ be a closed oriented Riemannian surface of positive genus with Gauss curvature  $K_g \in \{-1,0\}$, then there exists $\delta > 0$ such that	for every critical point $u$ of $E$ 
	$$
	 E(u) \geq \frac{4\pi}{3} + \delta. 
	$$
\end{theorem}

The novel aspect of this result is that we are comparing the energy of the standard bubble to the energy of critical points $u$ on surfaces with different topologies.

\begin{remark}
It follows easily from our proof of this result that the constant $\delta$ depends only on the genus and the injectivity radius, where in the case of flat tori we additionally ask that the metric has unit area.  	
\end{remark}

On the other hand, given \emph{any} closed oriented Riemannian surface $\Sigma$,  one can easily construct a sequence of maps $\{u_k\} \In \dot H^1(\Sigma,\R^3)$ for which $E(u_k)\to \frac{4\pi}{3}$ and $\ed E(u_k)\to 0$ in $H^{-1}$ (or even in $L^2$, e.g. pick a sequence $\{u_k\}\In\mathcal{Z}$ defined as in Section \ref{sec:bubble-space} so that $\la_k=\la(z_k) \to \infty$).

The bubbling analysis of Brezis-Coron \cite{BC85}, see Appendix \ref{sec:PS}, implies that any such Palais-Smale sequence 
subconverges to a bubble tree consisting of a trivial base map and a bubble and it is hence natural to ask whether there is a relation between the rate at which the energy concentrates, the distance of these maps to the bubble tree (defined in a suitable way) and the rate at which $\ed E(u_k)$  tends to zero.

The corresponding question for almost critical points that are defined on $\Sp^2$ and that converge to a solution of \eqref{eq:Heq} on $\Sp^2$ is addressed by the seminal results of Simon \cite{Simon} which yield that maps close to the set $\BB^1_{\Sp^2}$ of degree one conformal maps from $\Sp^2$ to $\Sp^2$
satisfy \Loj-inequalities of the form
\beqs
\dist_{\dot H^1}(u,\BB^1_{\Sp^2})\leq C\norm{\ed E(u)}_{L^2(\Sp^2)} \text{\quad and  \quad}
 \left|E(u)-\tfrac{4\pi}{3}\right|\leq C \norm{\ed E(u)}_{L^2(\Sp^2)}^2.
\eeqs
 We note that the reason for the optimal exponents in the above inequalities is that there are no non-trivial Jacobi fields, compare
 \cite[Lemma 9.2]{CM05}.

If we however want to analyse a sequence of almost critical points $u_k$ of an energy $\II$ that are defined on a domain $\Si_0$ but that converge in some sense to a limiting map $u_\infty$ defined on a domain $\Si_\infty$ of different topology, the classical \Loj-Simon inequalities \cite{Simon} do not apply.

In the important works of Topping \cite{Top04A}, Colding-Minicozzi \cite{ColdMini15} and Chodosh-Schulze \cite{ChSchulze} the major problem of changing topology has been addressed using distinct approaches. In \cite{Top04,Top04A} Topping performs a very careful bubbling analysis for nearly-harmonic maps between spheres which exploits the holomorphic-antiholomorphic decomposition of the energy to establish optimal \Loj-inequalities and precise bubble-scale estimates near resulting bubble trees. 

Colding and Minicozzi \cite{ColdMini15} take a different approach to establish uniqueness of blow-ups for generic, and hence cylindrical, singularities of mean curvature flow. They prove \Loj-inequalities valid for all submanifolds which are close to and graphical over a cylinder on a large ball, and combine these with a subtle analysis of the asymptotic behaviour of the flow to establish fast decay of the error terms appearing in these \Loj-inequalities. 

In \cite{ChSchulze} Chodosh-Schulze first use a dimension-reduction technique to obtain a \Loj-inequality for entire graphs over a conical self shrinker of mean curvature flow. Via localisation they then establish \Loj-Simon inequalities that allow them to prove uniqueness of blow-ups for conical singularities of mean curvature flow.

In the present paper we use a somewhat different approach to analyse sequences $u_k:\Sigma \to \R^3$ bubble converging to a trivial body map and a single bubble $u_\infty:\Sp^2 \to \R^3$ which is inspired in particular by the works of Isobe \cite{T00} and of Chanillo and the first author \cite{CM05} on asymptotic analysis for the $H$-surface energy with small boundary data on two-dimensional domains.  

Rather than altering the $\{u_k\}$'s to view them as maps on the limiting domain and exploiting a \Loj-inequality thereon, we carefully modify the collection of possible limits $\{u_\infty:\Sp^2\to \R^3\}$ to obtain a finite dimensional manifold $\ZZ$ of \textit{adapted bubbles} which are defined on the original surface and which reflect the possible blow-up behaviour. 

The obvious disadvantage of working on the original domain $\Sigma$ is that these adapted bubbles will not be critical points, so we cannot use the approach of Simon to prove \Loj-inequalities. We instead prove \Loj-inequalities for maps which are $\dot H^1$-close to this manifold of adapted bubbles. 

The advantage of this approach is that we will not need to modify 
the sequence of almost critical points  $u_k:\Si\to \R^3$ that undergoes bubbling 
and indeed will require very little information about the behaviour of sequences of maps $u_k:\Si\to \R^3$. Instead, we shall see that the key assumptions required in our abstract results, which we state and prove in Section \ref{s:abstract}, 
concern only the behaviour of the energy and its variation on the, in our case explicitly constructed, set of adapted bubbles. 
As we are concerned with the study of maps that approach the first non-trivial bubble tree
the set of adapted bubbles $\ZZ:=\{Rz_{a,\la}\}$, which will be defined in detail in Section \ref{sec:bubble-space}, is described by their concentration point 
$a\in \Si$, a rotation $R\in SO(3)$ in the target and a (large) scaling parameter $\lambda$.

Our main result can then be formulated as follows:

\begin{theorem}\label{thm:main}
	Let $(\Sigma, g)$ be a closed oriented Riemannian surface of positive genus with Gauss curvature  $K_g \in \{-1,0\}$ and let $\mathcal{Z}$ be the set of adapted bubbles defined in Section \ref{sec:bubble-space}.  Then there exist $\eps>0$, $\la_*>1$ and $C\in \R$ so that the following holds: 
	Suppose $u\in \dot H^1(\Sigma,\R^3)$ satisfies  
	$${\rm{dist}}_{\dot H^1}(u,\mathcal{Z}_{\la_*}) <\eps,$$ 
	where $\ZZ_{\la_*}$ is the subset of $\ZZ$ consisting of adapted bubbles with scaling parameter $\la(z)\geq \la_*$.
	
	Then for any $z\in \mathcal{Z}$ with ${\rm{dist}}_{\dot H^1}(u,\mathcal{Z}) = \|u-z\|_{\dot H^1}$ the bubble scale is so that
	\beq \label{est:mainloj1}
	\lambda(z)^{-1}\leq C(1+|\log \|\ed E(u)\|_{L^2(\Si,g)}|^\fr12) \|\ed E(u)\|_{L^2(\Si,g)},
	\eeq
	the distance of $u$ from the adapted bubble set is bounded by
	\beq  \label{est:mainloj2}
{\rm{dist}}_{\dot H^1}(u,\mathcal{Z}) 
\leq C\|\ed E(u)\|_{L^2(\Si,g)}\eeq
and we have a gradient \Loj-inequality of the form 
\beqs
\left|E(u)-\frac{4\pi}{3}\right|\leq C (1+|\log \|\ed E(u)\|_{L^2(\Si,g)}|) \|\ed E(u)\|_{L^2(\Si,g)}^2.
\eeqs
\end{theorem}
We note that, for $u$ satisfying the hypotheses of the above theorem (for suitably chosen $\eps$ and $\la_*$) there always exists a $z\in \ZZ$ minimising $\dist_{\dot H^1}(u,\ZZ)$, and any such $z$ will satisfy $\la(z)\geq \la_*/2$, see Remark \ref{rmk:distance}.

We also have the analogous $H^{-1}$ version of the above which has content for the $H$-functional due to the bubbling theory of Brezis-Coron \cite{BC85} (cf Appendix \ref{sec:PS}).\begin{theorem} \label{thm:Loj-H-1}
Under the hypotheses of Theorem \ref{thm:main} we can furthermore estimate 
$$
\la^{-1}(z)\leq C\norm{\ed E(u)}_{H^{-1}}^\half \text{ and } {\rm{dist}}_{\dot H^1}(u,\mathcal{Z}) \leq C\norm{\ed E(u)}_{H^{-1}}(1+\abs{\log\norm{\ed E(u)}_{H^{-1}}}^{\half})$$
and have a \Loj-type estimate involving the $H^{-1}$ norm of the form 
\beqas 
\left |E(u) -\frac{4\pi}{3} \right|&\leq C\norm{\ed E(u)}_{H^{-1}}.
\eeqas
\end{theorem}

\begin{remark}\label{rmk:althyp}
	Theorem \ref{thm:PS} ensures that if $\Sigma \neq \Sp^2$ then any Palais-Smale sequence $\{u_k\}$ with $E(u_k)\to \frac{4\pi}{3}$ satisfies the hypotheses of Theorem \ref{thm:main} for sufficiently large $k$ (see Remark \ref{rmk:A2}). Since the conclusions of the above theorems are trivially satisfied when $\ed E$ is large in the corresponding norm (see also \eqref{eq:E16}), we can replace the hypothesis that $\dist_{\dot H^1}(u,\ZZ)<\eps$ by 
	\beqs
	\left |E(u) -\frac{4\pi}{3} \right|<\tilde{\eps}
\eeqs
for a suitable $\tilde{\eps}>0$. 
\end{remark} 
We will be able to derive our main results on the $H$-surface equation from an abstract result, stated in Theorem \ref{thm:abstract-main},  about \Loj-inequalities in neighbourhoods of suitable manifolds of almost critical points that we prove in Section \ref{s:abstract} using a finite-dimensional reduction argument. 

\

To be able to apply Theorem \ref{thm:abstract-main} it is crucial that the set of adapted bubbles is constructed very carefully. The details of this construction are carried out in Section \ref{sec:bubble-space} but here we briefly describe the main ideas in the 
 case where the domain is a
flat torus. We let  $\bl:\R^2\to \Sp^2$ denote the \emph{standard bubble}, which is the inverse stereographic projection 
\begin{equation}\label{eq:bubblePi}
\bl(x) = \left(\frac{2x}{1+|x|^2} , \frac{1-|x|^2}{1+|x|^2}\right), 
\end{equation}
 with $x = (x_1,x_2)$ and recall \cite[Lemma A.1]{BC85} that the set of solutions of \eqref{eq:Heq} on $\R^2$ with energy $\frac{4\pi}{3}$ is given (upto additive constant) by 
 \beq \label{def:B1}
 \mathcal{B}^1:=\{R \bl_{\lambda,a}(x), a\in \R^2, \lambda >0, R\in SO(3)\} \text{ for } \bl_{\lambda,a}(x):=\bl(\lambda (x-a)).
 \eeq

Bubbling analysis, compare Appendix \ref{sec:PS}, implies that any Palais-Smale sequence $\{u_k\}$ with $E(u_k)\to \frac{4\pi}{3}$ on the torus, viewed as maps on a suitable fundamental domain, shadows 
a sequence of the above bubbles $R_k\bl_{\lambda_k,a_k}$ with $\lambda_k\to \infty$, see Theorem \ref{thm:PS} and Remark \ref{rmk:A2}. 

 In order to compare such a sequence $u_k$ with the corresponding bubble, we want to adapt the bubbles to obtain maps $Rz_{\la,a}$ from the torus. 
To do this we use that 
 for $\lambda |x-a|$ large 
$$
  \bl_{\lambda,a} =  \left(\frac{2}{\lambda} \frac{x-a}{|x-a|^2},-1\right)  + O \left( \frac{1}{\lambda^2 |x-a|^2} \right). 
$$ 
Up to lower-order terms, the first two components of $ \bl_{\lambda,a}$ behave like a multiple of the gradient of the fundamental solution
of the Laplacian 
in $\R^2$ and therefore it is natural to adjust them globally on the 
surface to a suitable multiple of the \emph{gradient of the Green's function} $G$ on $\Si$, which is characterised by
 \begin{equation}\label{eq:Grn}
 -\Dl_p G(p,a) = 2\pi\dl_a - \frac{2\pi}{\Vol_g(\Sigma)}.	
 \end{equation}

This construction is inspired by other geometric problems with loss of compactness such as Liouville's 
equations in two dimensions \cite{BP98}, and the Yamabe equation 
in dimension greater or equal to three  \cite{Schoen84,Rey90}. The use of the 
Green's function is the most convenient from the energetic point of view, since it affects the Dirichlet 
energy much less than a naive cut-off to a constant function on the closed surface.

 In our case this construction, which is carried out in detail in Section \ref{sec:bubble-space}, will yield a six dimensional submanifold of adapted bubbles $\{R z_{\la,a}\}$ which are so that $z_{\la,a}\simeq \bl_{\la,a}(x)$ in suitable isothermal coordinates. One of the crucial properties of this set of adapted bubbles is that we have the energy expansion 
 $$
 E(z_{\lambda,a}) = \frac{4 \pi}{3}- \frac{4\pi}{\lam^2} \mathcal{J}(a) + O(\la^{-3}), \text{ for }  \mathcal{J}(a)=  \lim_{x\to 0} \left((\partial_{y_1}\partial_{x_1}+\partial_{y_2}\partial_{x_2}) G_a(x,y)|_{y=0}\right), 
 $$
where $G_a$ represents the Green's function in suitable local isothermal coordinates centred at $a$, 
see Lemma \ref{lemma:principal-term}.   This expansion is similar  to those appearing in \cite{T00,CM05} 
on Euclidean domains, however we are able here to fully characterise the sign of $\mathcal{J}$,
which is a key feature. While $\mathcal{J}\equiv 0$ on $\Sp^2$, we can show that $\mathcal{J}$ is strictly negative on all surfaces of positive genus.
For flat tori of unit area one finds that $\mathcal{J}(a) = -2\pi$ for all $a$ which follows from a simple calculation carried out at the beginning of Section \ref{subsec:Bergman}. For higher genus surfaces $(\Sigma,g)$ equipped with a hyperbolic metric we shall prove in Section \ref{subsec:Bergman} (see Remark \ref{rmk:Jsign}) that 
$$\mathcal{J}(a)=-8\pi \sum_{j} |\phi_j(a)|_g^2$$whenever $\{\phi_j\}$ is an $L^2$-orthonormal basis of holomorphic one forms. A consequence of Riemann-Roch is that there is no point on $\Sigma$ where all holomorphic 
differentials vanish simultaneously, see for example Corollary 11 in \cite{Bobenko}. Thus we may conclude that $\sup_a\mathcal{J}(a) = \mathcal{J}_{(\Sigma,g)} <0$ for any surface of positive genus. For the success of our method it is crucial that $\mathcal{J}$ has a definite sign, as this forces the energy decay in the $\la$-direction to dominate all other variations appearing in our dimension-reduction.   

The approach utilised in this paper has since been generalised to the study of low-energy harmonic maps from closed surfaces into closed target manifolds by the second author \cite{Rup21, Rup23}, where in particular energy gap and \Loj-type estimates are obtained in this more complicated setting, with applications to convergence rates of the harmonic map heat flow into analytic manifolds.

\

This paper is outlined as follows. We begin by developing a framework to prove suitable \Loj-inequalities near manifolds of almost critical points of functionals on Hilbert spaces whose second variations is definite on the manifold in orthogonal directions. The rest of the paper is concerned with almost critical points of the $H$-energy. In Section \ref{sec:main-proofs} we carefully construct the manifold of adapted bubbles on higher-genus surfaces and state the required results concerning definiteness and asymptotic behaviour of the energy $E$ in order to apply our abstract results and conclude the proofs of our main Theorems \ref{t:gap} and \ref{thm:main}. The proof of these properties is carried out in Sections \ref{sec:proofs-of-lemmas} and \ref{subsec:Bergman} and we recall the required bubbling analysis of Palais-Smale sequences in the Appendix. In Section \ref{s:Wente} we exploit the relationship between the $H$-energy and the classical Wente inequality to obtain \Loj-inequalities for functions with near-optimal Wente energy on surfaces with positive genus.

\

\begin{center}
Acknowledgements	
\end{center}

The authors are grateful to Marco Bertola, Francesco Bonsante, Andrea di Lorenzo and Angelo Vistoli 
for useful comments and references on Algebraic Geometry. A.M. has been partially supported by the project {\em Geometric problems with loss of compactness}  from Scuola Normale Superiore and by MIUR Bando PRIN 2015 2015KB9WPT$_{001}$ and is also a member of
GNAMPA as part of INdAM.


\section{An abstract result}\label{s:abstract}

In this section we prove an abstract result concerning 
\Loj-inequalities for elements that are close to a set $\ZZ$ of approximate critical points for functionals on a general Hilbert-space, in settings that occur if the underlying singularity models form a non-degenerate manifold of critical points. 
We remark that there is an extensive literature on \Loj-inequalities for maps  near \textit{critical points} of functionals on  Hilbert spaces, beginning with the seminal work of Simon \cite{Simon}, and more generally on problems that can be studied via  finite-dimensional reductions see e.g. \cite{AM06}, the references therein and the discussion below.

Let $\mathcal{H}$ be a Hilbert-space with inner product $\lan\cdot,\cdot\ra_\HH$,
let $\ZZ$ be a finite dimensional submanifold of $\HH$, let $\lambda:\ZZ\to (0,\infty)$ be a continuous function and denote by $\ZZ_{\la_0}$ the set of elements of $\ZZ$  with $\la(z)\geq \la_0$. We furthermore let $\VV_z$, $z \in \mathcal{Z}$,
be the orthogonal complement 
in $\mathcal{H}$ of $T_z \mathcal{Z}$, and denote by $P^{\VV_z}$ the orthogonal projection from $\mathcal{H}$ 
to $\mathcal{V}_z$.

We consider an energy functional
$\mathcal{I} : \mathcal{H} \to \R$ of class $C^2$ and want to identify conditions on the behaviour of $\II$ on $\ZZ$ that are sufficient to establish \Loj-type inequalities on a suitable neighbourhood of $\ZZ_{\la_0}$ for some $\lam_0>1$. To begin with we ask
\begin{ass}\label{ass:A}
	The second variation $\ed^2 \II(z)$, $z\in \ZZ_{\lambda_0}$, is uniformly definite orthogonally to $\ZZ$ in the sense that 
		there exists $c_0>0$ so that for every $z\in\ZZ_{\lambda_0}$ we can split 
	 $\VV_z=\VV_z^+\oplus \VV_z^-$ orthogonally 
	 so that for any $w^\pm\in \VV_z^\pm$ we have 
$\ed^2 \II(w^+,w^-)=0$
while 
$$\pm \ed^2 \II(z)[w^\pm,w^\pm]\geq c_0 \norm{w^\pm}_\h^2.$$
\end{ass}

\begin{remark}\label{rmk:equiv-ev-gap}	
For many energies, including the $H$-energy considered in this paper, this condition can alternatively be formulated in terms of spectral properties of  the Jacobi operator $L_z$, $z\in\ZZ$,  characterised by 
 $$\ed^2 \II(z)[v,w]=\langle L_z(v),w\rangle_\h \text{ for all } v,w\in \h.$$
 Namely if the 	
 projected Jacobi operator $\hat L_z:=P^{\VV_z}\circ L_z$
   admits an orthonormal eigenbasis, as we shall see is the case for the $H$-energy, then the above assumption is equivalent to a uniform eigenvalue gap near zero with the splitting simply given by the orthogonal splitting of $\VV_z$ into the subspaces $\VV_z^\pm$ spanned by the eigenfunctions with positive respectively negative eigenvalues.
 \end{remark}

In addition to this non-degeneracy we require the following conditions on the behaviour of $\II$ and its first and second variation at points of the manifold $\ZZ_{\la_0}$

\begin{ass} \label{ass:B}
There are non-increasing functions $f_{0,1,2}:[\lambda_0,\infty)\to (0,\infty)$ and a number $\nuu\geq 1$ so that 
\beq \label{ass:quotient}
\frac{f_1(\lambda)f_2(\lambda)+f_2(\la)^{\nuu}}{f_0(\lambda)} \to 0 \text{ as } \lambda \to \infty
\eeq
for which the following holds.
	For each $z\in \ZZO$ there exists $y_z\in \HH$, normalised to $\norm{y_z}_\HH=1$, so that
\beq
\label{ass:f0}
\abs{\ed \II(z)(y_z)}\geq f_0(\lambda(z))
\eeq
while
\beq \label{ass:f1}
\norm{(\ed^2 \II)(z)(y_z,\cdot )}_{(\VV_z)^*}\leq f_1(\lambda(z))
\eeq
and 
\beq\label{ass:f2}
\norm{\ed \II(z)}_{(\VV_z)^*}\leq f_2(\lambda(z))\eeq
and so that for all $w\in \VV_z$
\beq 
\label{ass:nuu} 
\babs{\left(\ed^2 \II(z+tw)-\ed^2 \II(z)\right)[y_z,w]}\leq C \norm{w}^{\nuu}_\HH, \quad \text{ for all } t\in [0,1].
\eeq
\end{ass}

We stress that the above assumptions only need to be satisfied for elements of $\ZZ$ and the specific element $y_z$, and note that in applications
the validity  of \eqref{ass:quotient}-\eqref{ass:nuu} is very sensitive to the construction of $\ZZ$ and choice of $y_z$. In contrast, at general points in a neighbourhood of $\ZZ_{\la_0}$ we require only the following basic properties of the energy. 

\begin{ass}
\label{ass:C}
 There exist $\delta>0$, $C\in \R$ and a modulus of continuity $\omega(\cdot)$ 
near zero such that for all  $u_{1,2}\in \HH$ with $\dist_\HH(u_{1,2},\ZZO)<\delta
$ we have 
\beq \label{ass:general-dI}
\|  \ed^2 \II(u_1)  \| \leq C\eeq
and 
\beq\label{ass:generald2I}
 \|  \ed^2 \II(u_1)  - \ed^2 \II(u_2)  \| \leq 
\omega(\|u_1-u_2\|_\h).\eeq 
\end{ass}

For the $H$-energy we will see that assumptions \eqref{ass:nuu}-\eqref{ass:generald2I} are trivially satisfied with $\nu=2$ and $\om(t)=Ct$, see \eqref{est:D2E-H} below.

\

Our first abstract result excludes the possibility that \textit{critical points} of the energy $\II$ approach the manifold $\ZZ$ and would be sufficient to establish Theorem \ref{t:gap} on the existence of an energy gap for critical points of the $H$-energy.

\begin{theorem}\label{thm:abstract-critical-points}
Let $\II:\HH\to \R$ be a $C^2$ functional of a Hilbert space $(\HH,\langle\cdot,\cdot\rangle_\HH)$ for which Assumptions \ref{ass:A}, \ref{ass:B} and \ref{ass:C} hold true for some finite dimensional submanifold $\ZZ$ and a continuous $\lambda:\ZZ\to (0,\infty)$. Then there cannot be any sequence of critical points $u_k$ of $\II$ for which there are $z_k\in\ZZ$ with $\lambda(z_k)\to \infty$ and 
$\dist_{\h}(u_k,\ZZ)=\norm{z_k-u_k}_{\h}\to 0$.
\end{theorem}

Dimension reduction techniques have been used successfully to prove several results about the structure of sets of critical points, see e.g. \cite{AM06} and the reference therein.

One of the key ideas in these results is that if the second variation 
of $\II$ is invertible orthogonally to a given set of approximate critical points $\ZZ$ then  
moving orthogonally to $\ZZ$ one can control the projection of the gradient onto the normal space, also called the \textit{auxiliary equation}. If it is also possible, e.g. by variational or topological means, to solve the so called \textit{bifurcation equation}, which describes the tangential component of the gradient, one can then obtain the existence of critical points close to the given set of almost critical points. 

While the auxiliary equation is always solvable, in our settings the bifurcation equation is strictly not solvable in an asymptotic sense, which could be used to establish the above result.

Since our goal is to not only exclude the existence of critical points, but to also obtain quantitative information for all maps in a neighbourhood of $\ZZ$, we need a different approach which will at the same time also give the proof of the above result as a by product.  

\

When proving \Loj -inequalities it is common that one wants to bound quantities using the gradient of the functional with respect to a norm which is different from that of $\HH$. For geometric energies such the $H$-energy the underlying space consists of $H^1$ functions but in applications it might be more useful to have a \Loj -inequality that involves the $L^2$-gradient rather than the $H^{-1}$ norm of $\ed E$. 

We hence introduce a further norm $\norm{\cdot}_{*}$ on $\HH$  and 
ask that additionally 

\begin{ass}\label{ass:D}
There exists a non-increasing function $f_3:[\la_0,\infty)\to \R$ so that for every $z\in \ZZ_{\la_0}$ the element $y_z$ from Assumption \ref{ass:B} satisfies
\beqs 
\norm{y_z}_{*}\leq f_3(\la(z))\eeqs 
with respect to a norm $\norm{\cdot}_{*}$ on $\HH$ which is furthermore so that there exists some $C\in\R$ so that 
$$\norm{v}_{*}\leq C\norm{v}_{\HH}\text{ for every } v\in \HH.$$
\end{ass}

These assumptions turn out to be sufficient to bound the distance of a map $u$ to $\ZZ$ in terms of  $\norm{\ed \II(u)}_*$, where here and in the following we abuse notation by writing both $\norm{\cdot}_{*}$ for the additional norm on $\HH$ as well as for the norm on the dual space of $(\HH,\norm{\cdot}_*)$.  

Finally, we want to prove a gradient \Loj-inequality, where we compare the energy of a map $u$ with the energy of a nearby critical point. As the elements of $\ZZ$ are not critical point themselves, we do not just want to bound the difference of $\II(u)$ and $\II(z)$, but instead want to compare 
$\II(u)$ with the energy $\II^*$ of the underlying singularity model or bubble tree.  We hence furthermore need to know how the energy of elements of $\ZZ$ compares to the energy of the underlying critical point and thus assume:
\begin{ass}\label{ass:E}
There exists $\II^*\in \R$ and a non-increasing function $f_4:[\la_0,\infty)\to \R$ so that for all $z\in \ZZ_{\la_0}$
$$\abs{\II(z)-\II^*}\leq f_4(\la(z)).$$
\end{ass}

To state our main result we will restrict our attention to functions $f_{i}$ with decay 
\beq\label{ass:polyn-f}
	f_0(\la)\geq c_1\la^{-\gamma_0}(\log \la)^{-\si_0}
	 \text{ while } \quad f_i(\la)\leq C\la^{-\gamma_i}(\log \la)^{-\si_i}%
	\text{ for } i=1,\ldots, 4
	\eeq
for some $c_1>0$, $C\in \R$ and numbers $\gamma_{i}\geq 0$ as well as real numbers $\si_{i}$, which for $i\neq 1,3$ are asked to be positive if the corresponding $\gamma_{i}=0$. 

\begin{remark}
We note that the explicit decay in \eqref{ass:polyn-f} is needed only to make the following easier to read and that the presented argument can be used also to derive similar results in more general settings.	
\end{remark}

Our main result in the abstract setting of functionals on Hilbert-spaces can then be formulated as follows. 
\begin{theorem}\label{thm:abstract-main}
Let $\II:\HH\to \R$ be a $C^2$ functional on a Hilbert space $(\HH,\langle\cdot,\cdot\rangle_\HH)$ for which Assumptions \ref{ass:A}-\ref{ass:E} hold true for some finite dimensional submanifold $\ZZ$, a continuous $\lambda:\ZZ\to (0,\infty)$ and for functions $f_{i}$ as in 
\eqref{ass:polyn-f} for numbers $\gamma_i\geq 0$ and $\si_i$ as described above. 

Then there exists $\eps_1,c,C>0$ and $\la_*\ge \la_0$ so that for every $u\in \HH$ with 
$$\norm{u-z}_\h=\dist_\h(u,\ZZ)<\eps_1 \text{ for some } z\in \ZZ_{\la_*}$$ 
and $\norm{\ed \II(u)}_{*}\leq \half$ the following holds true: 
\begin{enumerate}
\item 
In the case that $\gamma_0>0$ we can bound
\begin{eqnarray}
\label{claim:la-thm-abstract}
\la(z)^{-1}&\leq &C \norm{\ed \II(u)}_{*}^{\al_1}\abs{\log\norm{\ed \II(u)}_{*}}^{\beta_1}\\
\label{claim:Loj1-thm-abstract-general}
\norm{u-z}_{\HH}&\leq &C\norm{\ed \II(u)}_{*}^{\al_2}\abs{\log\norm{\ed \II(u)}_{*}}^{\beta_2} \\
\label{claim:Loj2-thm-abstract}
\abs{\II(u)-\II^*}&\leq &C\norm{\ed \II(u)}_{*}^{\al_3}\abs{\log\norm{\ed \II(u)}_{*}}^{\beta_3} 
\end{eqnarray}
for exponents $\al_i> 0$ and $\beta_i\in \R$ defined in Remark \ref{rmk:coeff-thm-abstract} below.
\item If instead $\gamma_0=0$ then we have 
\begin{eqnarray}
 \label{claim:la-thm-abstract-exp}
\la(z)^{-1}&\leq& \exp(-c\norm{\ed \II(u)}_{*}^{-\eta_1})\\
\label{claim:Loj1-thm-abstract-exp}
\norm{u-z}_{\HH}&\leq &C\norm{\ed \II(u)}_{*}^{\eta_2}\\
\label{claim:Loj2-thm-abstract-exp}
\abs{\II(u)-\II^*}_{\HH}&\leq &C\norm{\ed \II(u)}_{*}^{\eta_3}
\end{eqnarray}
for exponents $\eta_{1,2,3} >0$ defined in Remark \ref{rmk:coeff-thm-abstract-exp} below.
\end{enumerate}
\end{theorem}

To define the above exponents 
we will write  $\FF_\infty\big((a_1,b_1), (a_2,b_2)\big)$ for the `optimal' pair of exponents $(\tilde a,\tilde b)$ 
satisfying 
$$\la^{a_1}(\log \la)^{b_1}+\la^{a_2}(\log \la)^{b_2}\leq C \la^{\tilde a}(\log \la)^{\tilde b} 
\quad\text{ for all sufficiently large } \la,$$
i.e. for $\tilde a:=\max\{a_1,a_2\}$ and for $\tilde b$ given by $\tilde b=b_1$ if either $a_1>a_2$ or if $a_1=a_2$ and $b_1\geq b_2$ respectively $\tilde b=b_2$. We also let $\FF_0\big((a_1,b_1), (a_2,b_2)\big)$ be the exponents that are characterised by the analogue property for small positive numbers.

As we shall see the exponents $(\al_i,\be_i)$ depend on an upper bound on $f_1+f_3$ rather than on the individual bounds on $f_1$ and $f_3$ so we set 
$(\gamma_{1,3},\si_{1,3}):=-\FF_0\big(-(\gamma_1,\si_1), 
-(\gamma_3,\si_3))$, i.e. choose the optimal exponents so that $f_1+f_3\leq C \la^{-\gamma_{1,3}}(\log \la)^{-\si_{1,3}}$.

 The exponents for which Theorem \ref{thm:abstract-main} is valid can then be computed as follows
\begin{remark}\label{rmk:coeff-thm-abstract}
If $\gamma_0>0$ the exponents in the above theorem are characterised by 
\beqa\label{def:gamma-sigma-thm}
\big(\tfrac{1}{\al_1}, \tfrac{\beta_1}{\alpha_1}\big)&:= \FF_\infty\big((\gamma_0-\gamma_{1,3}, \si_0-\si_{1,3}), (\tfrac{\gamma_0}{\nuu}, \tfrac{\si_0}{\nuu})\big),\\
\big(\al_2,\beta_2)&:= \FF_0\big((1,0), (\gamma_2\al_1, \gamma_2\beta_1-\si_2)\big),\\
\big(\al_3,\be_3)&:=\FF_0\big((2\al_2,2\be_2), (\gamma_4\al_1,\gamma_2\be_1-\si_4)\big).
\eeqa 
\end{remark}

\begin{remark}\label{rmk:coeff-thm-abstract-exp}
For $\gamma_0=0$ the above theorem holds true for exponents $\eta_i>0$ chosen as 
\begin{itemize}
\item $\eta_1=\frac{\nuu}{\si_0}$ if $\gamma_{1,3}>0$ 
 while $\eta_1=\big(\max\{\si_0-\si_{1,3}, \frac{\si_0}{\nuu} \}\big)^{-1}$ otherwise, 
\item $\eta_1=1$ if $\gamma_2>0$ while $\eta_2=\min\{1, \si_2\eta_1\}$ otherwise,
\item  $\eta_3=2\eta_1$ if $\gamma_4>0$ while 
$\eta_3=\min\{2\eta_2,\si_4\eta_1\}$ otherwise.
\end{itemize}

\end{remark}

\subsection{Proofs of the abstract results}
We first use the definiteness of the second variation in orthogonal directions to show
\begin{lemma}\label{lemma:abstract1}
Suppose that Assumptions \ref{ass:A} and \ref{ass:C} are satisfied and let  $\eps_1>0$ be so that $\omega(\eps_1)\leq \half c_0^2$. Then for all maps $u\in \HH$ which are so that there exists $z\in \ZZ_{\lambda_0}$  with 
$$\norm{u-z}_{\HH}=\dist_{\HH}(u,\ZZ)<\eps_1 $$  we can bound 
\beqs
\norm{u-z}_{\HH}\leq C \norm{\ed \II(u)}_{(\VV_z)^*}+C\norm{\ed \II(z)}_{(\VV_z)^*}.
\eeqs
\end{lemma}
\begin{proof}
By construction $w=u-z$ is orthogonal to $T_z\ZZ$ so we can split  $w=w^++w^-\in \VV_z^+\oplus \VV_z^-$, set $\tilde w:=w^+-w^-$ 
 and use Assumption \ref{ass:A} to bound
\beqas
\ed^2 \II(z)[w,\tilde w]=\ed^2 \II(z)[w^+,w^+]-\ed^2 \II(z)[w^-,w^-]\geq c_0^2(\norm{w^+}_{\HH}^2+\norm{w^-}_{\HH}^2)=c_0^2\norm{w}_{\HH}^2.\eeqas

Setting $u_t=z+tw$, $t\in[0,1]$, we then note that
$$\ed \II(u)[\tilde w]-\ed \II(z)[\tilde w]=\int_0^1\ed^2 \II(u_t)[w,\tilde w]\id t=\ed^2 \II(z)[w,\tilde w] - \text{err}_1\geq\half c_0^2\norm{w}_{\HH}^2$$
where we use in the last step that Assumption \ref{ass:B} and the choice of $\eps_1$ imply 
$$\text{err}_1=\int_0^1 \big(\ed^2 \II(z)-\ed^2 \II(u_t)\big)[w,\tilde w] \id t \leq \omega(\eps_1)\norm{w}_{\HH}\norm{\tilde w}_{\HH}\leq\half c_0^2\norm{w}_{\HH}^2.$$
\end{proof}

As a next step in the proof of both Theorems \ref{thm:abstract-critical-points} and \ref{thm:abstract-main} we now show 
\begin{lemma}\label{lemma:abstract2}
Suppose that Assumptions \ref{ass:A}, \ref{ass:B} and \ref{ass:C} are satisfied and let  $\eps_1>0$ be so that $\omega(\eps_1)\leq\half  c_0^2$. Then there exists $\la_1\geq \la_0$ and $C\in \R$  so that for every $u\in \HH$ with 
$\norm{u-z}_{\HH}=\dist_{\HH}(u,\ZZ)<\eps_1 $ for some $z\in \ZZ_{\la_1}$ we can bound 
\beqs
f_0(\la(z))\leq 
2\abs{\ed \II(u)(y_z)}+Cf_1(\la(z))\norm{\ed \II(u)}_{(\VV_z)^*}+C\norm{\ed \II(u)}^\nuu_{(\VV_z)^*}.
\eeqs
\end{lemma}
\begin{proof}
Using \eqref{ass:f0} we obtain, writing for short $\la=\la(z)$, $w=u-z$ and $u_t=z+tw$,
\beqas 
f_0(\la)&\leq \abs{\ed \II(z)[y_z]}\leq \abs{\ed \II(u)[y_z]}+\int_0^1\abs{\ed^2 \II(u_t)[y_z,w]} \id t\\
&\leq \abs{\ed \II(u)[y_z]}+ \abs{\ed^2 \II(z)[y_z,w]}+\sup_{t\in [0,1]} \abs{(\ed^2 \II(u_t)-\ed^2 \II(z))[y_z,w]}\\
&\leq \abs{\ed \II(u)[y_z]}+ f_1(\la)\norm{w}_{\HH}+C\norm{w}_{\HH}^\nuu,
\eeqas 
where we have used \eqref{ass:f1} and \eqref{ass:nuu} in the last step.
Combined with Lemma \ref{lemma:abstract1} we thus get
\beqas 
f_0(\la)&\leq \abs{\ed \II(u)[y_z]}+ Cf_1(\la)\norm{\ed \II(u)}_{(\VV_z)^*}+C\norm{\ed \II(u)}_{(\VV_z)^*}^{\nuu}
 +C(f_1(\la)f_2(\la)+f_2(\la)^\nuu).
\eeqas
For $\la_1$ chosen sufficiently large, we can use assumption \eqref{ass:quotient} to absorb the final term above into the left hand side yielding the desired estimate. 
\end{proof}

We note that Lemma \ref{lemma:abstract2} immediately implies that there cannot be any critical point $u$ so that $\dist_{\HH}(u,\ZZ)=\|u-z\|_\h<\eps_1$ for some $z\in \ZZ_{\la_1}$ which establishes Theorem \ref{thm:abstract-critical-points}. These lemmas furthermore allow us to prove our second result. 
While we carry out this proof in the general abstract setting below, for the convenience of the reader we also include this proof in the concrete setting of the $H$ functional, where the following computations are less technical, in Section \ref{sec:ill}.  
  
\begin{proof}[Proof of Theorem \ref{thm:abstract-main}] 
Let $u$ be so that the assumptions of the theorem are satisfied for some $\la_*\geq \la_1>1$ that is chosen below. We first combine Lemma \ref{lemma:abstract2} with  Assumption \ref{ass:D} to bound 
\beqas
f_0(\la) 
&\leq C \norm{d \II(u)}_{*} (f_1(\la)+ f_3(\la))+C\norm{\ed \II(u)}_{*}^\nuu.
\eeqas 
This allows us to conclude that  
\beqa \label{est:proving-lambda-bound1}
1&\leq C\la^{\gamma_0-\gamma_{1,3}}(\log \la)^{\si_0-\si_{1,3}}\norm{\ed \II(u)}_{*} +C\big(\la^{\gamma_0/\nuu}(\log \la)^{\si_0/\nuu}\norm{\ed \II(u)}_{*}\big)^\nuu.
\eeqa
In the case where $\gamma_0>0$ we thus obtain that  
\beqs 
1\leq C \max\{\la^{\gamma_0-\gamma_{1,3}}(\log \la)^{\si_0-\si_{1,3}}, \la^{\gamma_0/\nuu}(\log \la)^{\si_0/\nuu}\} 
\norm{\ed \II(u)}_{*}\leq C \la^{\frac{1}{\al_1}} (\log \la)^{\frac{\beta_1}{\al_1}}\norm{\ed \II(u)}_{*}\eeqs
 for $\al_1>0$ and $\beta_1\in\R$ defined in \eqref{def:gamma-sigma-thm}. 
 The resulting bound of 
\beq \label{est:la-proof-lemma-3-abstract2} \la^{-1}\leq C \norm{\ed \II(u)}_{*}^{\al_1}(\log \la)^{\beta_1}\eeq
immediately yields the claim \eqref{claim:la-thm-abstract} if $\beta_1=0$. If 
$\beta_1<0$ we additionally use that this bound implies that $\la^{-1}\leq C \norm{\ed \II(u)}_{*}^{\al_1}$ and hence $\log \la\geq c \abs{\log(\norm{d \II(u)}_{*})}$ and then combine this with  \eqref{est:la-proof-lemma-3-abstract2} to obtain the estimate \eqref{claim:la-thm-abstract} claimed in the theorem. Finally, if $\beta_1>0$, we note that \eqref{claim:la-thm-abstract} is trivially true if $\la\geq \norm{d \II(u)}_{*}^{- \alpha_1}$, while for $\la \leq  \norm{d \II(u)}_{*}^{-\alpha_1}$ we can 
 bound  $(\log \la)^{\beta_1}\leq C\abs{\log \norm{d \II(u)}_{*}}^{\beta_1}$ and combine this with \eqref{est:la-proof-lemma-3-abstract2} to obtain the claimed bound \eqref{claim:la-thm-abstract} on $\la^{-1}$. 
  
The analogue bound \eqref{claim:la-thm-abstract-exp} in the case that $\gamma_0=0$ can be immediately deduced from  \eqref{est:proving-lambda-bound1} which in this case implies that 
\beqs 1\leq C(\log \la)^{1/\eta_1}\norm{\ed \II(u)}_{*}\eeqs for $\eta_1$ as in Remark \ref{rmk:coeff-thm-abstract-exp}.

 To prove the \Loj -type estimates \eqref{claim:Loj1-thm-abstract-general} and \eqref{claim:Loj2-thm-abstract} in the case that $\gamma_0>0$ we will use that \eqref{claim:la-thm-abstract} implies 
\beq \label{est:fi-proof-abstract} 
\text{$f_i(\la)\leq \la^{-\gamma_i}(\log \la)^{-\si_i}\leq  C\norm{\ed \II(u)}_{*}^{\gamma_i\al_1}\abs{\log\norm{\ed \II(u)}_*}^{\gamma_i\beta_1-\si_i}$,  $i=2,4$.}\eeq
We can now combine this bound on $f_2$ with the 
 estimate 
\beq 
\label{est:w-proof-abstract}
\norm{u-z}_{\HH}\leq C\norm{\ed \II(u)}_{*}+Cf_2(\la)\eeq
obtained in Lemma \ref{lemma:abstract1} to deduce that the claimed  estimate \eqref{claim:Loj1-thm-abstract-general} holds for exponents $(\alpha_2,\beta_2)=\FF_0((1,0), (\gamma_2\al_1,\gamma_2\beta_1-\si_2))$. 
To derive the \Loj -estimate \eqref{claim:Loj2-thm-abstract} we then note that 
\beqa\label{est:2nd-Loj-abstract-proof}
\abs{\II(u)-\II^*}&\leq \abs{\II(u)-\II(z)}+\abs{\II^*-\II(z)}\leq 
 \abs{\ed \II(z)[w]}+\sup_{t\in [0,1]}\abs{(\ed \II(u_t)-\ed \II(z))[w]}+f_4(\la)\\
&\leq f_2(\la)\norm{w}_{\HH}+C\norm{w}^2_{\HH}+f_4(\la)
\leq C\norm{\ed \II(u)}_{*}^{2\al_2}\abs{\log\norm{\ed \II(u)}_{*}}^{2\si_2}+Cf_4(\la)
\eeqa
where we applied \eqref{claim:Loj1-thm-abstract-general} in the last step. 
Combined with \eqref{est:fi-proof-abstract} we obtain that indeed 
$$
\abs{\II(u)-\II^*}\leq C\norm{\ed \II(u)}_{*}^{\al_3}\abs{\log\norm{\ed \II(u)}_{*}}^{\si_3}
$$
for $(\al_3,\si_3)=\FF_0((2\al_2,2\si_2), (\gamma_4\al_4, \gamma_4\beta_1-\si_4))$. 

We finally return to the case that $\gamma_0=0$ where we know that $C\log \la\geq \norm{\ed \II(u)}_{*}^{-\eta_1}$. If  $\gamma_2>0$ then the second term in \eqref{est:w-proof-abstract} is exponentially small so we obtain that \eqref{claim:Loj1-thm-abstract-exp} holds true with optimal exponent $\eta_2=1$ while for $\gamma_2=0$ this estimate is valid with $\eta_2=\min\{1,\eta_1\si_2\}$.

Finally, from \eqref{est:2nd-Loj-abstract-proof} and the above we obtain \eqref{claim:Loj2-thm-abstract-exp}: 
$$\abs{\II(u)-\II^*}\leq  C \norm{\ed \II(u)}_*^{2\eta_2}+ Cf_4(\la)\leq C   \norm{\ed \II(u)}_*^{\eta_3}$$
where $\eta_3=2\eta_2$ if $\gamma_4>0$ while $\eta_3=\min\{2\eta_2,\si_4\eta_1\}$ otherwise. 

\end{proof}


\section{Proof of the main results on the $H$-energy}
\label{sec:main-proofs}

\subsection{Definition of the set of adapted bubbles}
\label{sec:bubble-space}

Let $\Sigma$ be a closed Riemann surface of positive genus $\gamma\geq 1$ which we equip with the unique compatible hyperbolic metric, respectively the unique flat metric for which $\Vol(\Sigma, g)=1$.

The goal of this section is to define the set of adapted bubble $\ZZ=\{Rz_{\la,a}\}$. As explained in the introduction, 
the idea is to define our $z_{\la,a}$ essentially as 
 an interpolation between a (highly-concentrated if $\lambda \gg1$) standard bubble $\pi_{\la}=\pi(\la\cdot)$ 
 based at $a\in \Sigma$, and a suitable multiple of the first derivatives of the Green's function on $(\Sigma,g)$ in the first two-components, and a simple cut-off in the third. 
 Here and in the following we continue to denote by $\pi$ the inverse stereographic projection $\pi(x)=\left(\frac{2x}{1+|x|^2} , \frac{1-|x|^2}{1+|x|^2}\right) $ and we want to define $z_{\la,a}$ in a way that we have $z_{\la,a}\approx \pi_{\la}(x)$ in   oriented isothermal coordinates $x=F_a(p)$ on $B_\iota(a)\In (\Si,g)$,  $\iota:=\half \text{inj}(\Si,g)$, which are as follows: 
\begin{remark} \label{rmk:def-Fa}
If $g$ is hyperbolic we set $\rho= \tanh(\iota/2)$ and use that  
for any $a\in \Si$ 
there exists an orientation preserving isometric isomorphism 
$$F_a:(B_{\iota}(a),g) \to \left(\mathbb{D}_{\rho},\tfrac{4}{(1-|x|^2)^2}g_E\right)$$ where $\mathbb{D}_{\rho}:=\{x\in \R^2: \abs{x}<\rho\}$ and $g_E$ is the Euclidean metric. 
 In the flat case we will always \emph{a-priori} pick a tiling of $\R^2$ which represents $\Sigma$, set $\rho=\iota$ and use the resulting euclidean translations $F_a$ to the origin as  coordinates.
\end{remark}
When the genus is larger than one, we note that $F_a$ is determined only up to 
a rotation, but that this will not affect the definition of the bubbles set, see Remark \ref{rmk:coord} below. 

\

For points in a ball $B_{2\iota}(q)\In (\Si,g)$ we can write the Green's function on a flat torus as 
  \begin{equation*}
 G(p,a) = - \log d_g(p,a) + J(p,a)
  \end{equation*}
while for hyperbolic surfaces
 \begin{equation*}
 G(p,a) = - \log \left(\tanh \tfrac{d_g(p,a)}{2}\right) + J(p,a)
  \end{equation*}
 for a smooth function $J_a:B_{r}(q)\times B_{r}(q)\to \R$ which we call the regular part of Green's function. 
In the coordinates  introduced in Remark \ref{rmk:def-Fa} Green's function and its regular part are described by 
\begin{equation}\label{eq:locG}
G_a(x,y):=G(F_a^{-1}(x), F_a^{-1}(y))= - \log |x-y| + J_a(x,y).
\end{equation}
In the flat setting $J_a(x,y) = J(F_a^{-1}(x), F_a^{-1}(y))$ however in the hyperbolic setting $J_a$ carries an extra term equal to $\log|x-y|- \log \left(\tanh \tfrac{d_g(F_a^{-1}(x), F_a^{-1}(y))}{2}\right)$ which is harmonic and quadratic at the origin. 
Hence we have
\beq
\label{eq:locG-deriv}
 \na_y G_a(x,0)=\frac{x}{\abs{x}^2}+ \na_y J_a(x,0)
\eeq
for the smooth function $J_a$ which represents the regular part of Green's function. We also note that 
 the conformal covariance of the Laplacian implies that $\na_y J_a(x,0)$ is harmonic. 

\

To construct our adapted bubbles we fix some $r\in(0,\iota/4)$ and a cut-off function 
$\phi\in C_c^\infty(\mathbb{D}_{2r},[0,1])$ with 
$\phi\equiv 1$ on $\mathbb{D}_r$. 
Given any $a\in \Si$ and  $F_a :B_{\iota}(a) \to \mathbb{D}_{\rho}$ a choice of isometry as in Remark \ref{rmk:def-Fa}, we define
 \begin{eqnarray}
\tilde{z}_{\lambda, a}(p) = \twopartdef{\hat{z}_{\lambda,a}(F_a(p))}{p\in B_{\iota}(a)}{\frac{2}{\lambda}(\de_{a^1}G(p,a), \de_{a^2} G(p,a),0)}{p\notin B_{\iota}(a)}	
\end{eqnarray}
where  $\de_{a^i} = (F_a^{-1})_{\ast} \de_{y^i}$ and where $\hat{z}_{\lambda, a}$ is given in the local coordinates $x=F_a(p)$ by
\beq\label{eq:zhat}
\hat{z}_{\lambda,a}(x) : = \phi(x) \left(\bl_{\lambda}(x) + \left(\tfrac{2}{\lambda}\D_y J_a(x,0), 1\right)\right) + (1-\phi(x)) \left(\tfrac{2}{\lambda} \D_y G_a(x,0),0 \right).
\eeq
We then set
\beqs
z_{\lambda,a} = \tilde{z}_{\lambda, a} - \dashint_\Sigma \tilde{z}_{\lambda ,a} \in \dot H^1(\Sigma,\R^3), 	
\eeqs
and define the full set of adapted bubbles by  
\begin{equation*}
\mathcal{Z}:=\{R z_{\lambda,a} | a\in \Sigma, R\in SO(3), \lambda >0\}. 	
\end{equation*}
As in the sequel we will be interested in adapted bubbles of sufficiently high concentration,  i.e. with  $\la \geq \mu \gg 1$, we furthermore set 
$\mathcal{Z}_\mu:=\{R z_{\lambda,a} | a\in \Sigma, R\in SO(3), \lambda \geq \mu\}. 	$

\begin{remark}\label{rmk:distance}
We note that our adapted bubble set $\ZZ$ is so that for each $a,b\in \R^+$,   
$\{z\in \ZZ: a\leq \la(z)\leq b\}$ is contained in a compact subset of $\dot H^1$. Furthermore we can chose $\eps>0$ small enough and $\la_*$ large enough so that for any $z\in \ZZ_{\la_*}$, $\|z-\tilde{z}\|_{\dot H^1} >2\eps$ for all $\tilde{z}\in \ZZ$ with $\la(\tilde{z}) \notin [\la(z)/2, 2\la(z)]$.    

Thus for suitable choices of $\eps$ and $\la_*$ in Theorem \ref{thm:main}, $\dist_{\dot{H}^1}(u,\ZZ)$ is attained by some $z\in \ZZ_{\la_*/2}$ for any $u$ satisfying the hypotheses. 
\end{remark}

\begin{remark}
\label{rmk:coord}
We finally note that the definition of the set of adapted bubbles is independent of the chosen isometry $F_a$ as a subset of the rotations $R$ naturally corresponds to rotations of the local coordinates in the domain. 
\end{remark}

If one wants to introduce local coordinates on the set of adapted bubbles when $\Sigma$ is hyperbolic, then one can use the following choice of isometries $F_a$ for points $a$ in a neighbourhood of a fixed $a_0\in \Si$.
\begin{remark}
Given a point $a_0\in \Si$, a tiling of the Poincar\'e disc $\mathbb{D}_1$ by fundamental domains of $\Sigma$ and a representative $\hat a_0\in \mathbb{D}_1$ of $a_0$
 which minimises the hyperbolic distance to the origin there is a canonical choice of 
 isometries $F_a$ for $a\in B_\iota(a_0)$:

 For all $ a\in B_\iota(a_0)$ we let $\hat a$ be the unique representative in $\mathbb{D}_1$ 
with distance less than $\iota$ to $\hat a_0$ and use that there 
 is a unique choice of hyperbolic translation $\tau_{\hat a}$ taking $\hat a$ to the origin. The local coordinate chart $F_a:B_\iota(a)\to \DD_\rho$ is then obtained by picking the representative $\hat p$ with distance less than $\iota$ from $a$ and then applying this hyperbolic translation, i.e. by setting $F_a(p)=\tau_{\hat a}(\hat p)$.
\end{remark}

\subsection{Proof of Theorems \ref{t:gap} and \ref{thm:main}}
In this section we explain how our main results on $H$-surfaces can be derived from the abstract results of Section \ref{s:abstract}. To this end we state a series of lemmas whose proofs are carried out in Section \ref{sec:proofs-of-lemmas}.

We first note that the $H$ energy can be written in a coordinate free form as
\begin{equation}\label{eq:defEcf}
E(u):= \frac12\int_{\Sigma} |\D u|^2 -\frac{1}{3} \int_\Sigma u\cdot \D u \dot \wedge \Db u.
\end{equation}
We use the notation $\Db b=(\ast_g \ed b)^\sharp$ so that $\Db b = e^{-2\varrho}(-\de_{x_2} b , \de_{x_1} b )$ in oriented local isothermal coordinates, and write
\begin{equation*}
\D w \dot\wedge \Db v  :=  \left( \begin{array}{c}
\D w^2\cdot\Db v^3 - \D w^3\cdot \Db v^2 \\
\D w^3 \cdot \Db v^1 - \D w^1 \cdot \Db v^3 \\
\D w^1 \cdot \Db v^2 - \D w^2 \cdot \Db v^3
\end{array}\right).  
\end{equation*} 
We note that this expression is \emph{symmetric} in $v$ and $w$ and recall that for $u,v,w\in \dot H^1(\Sigma,\R^3)$
\beq \label{eq:symm}
\int_{\Sigma} u\cdot \D w \dot \wedge \Db v = \int_{\Sigma} w\cdot \D u \dot \wedge \Db v .  \eeq

\begin{remark}\label{rmk:int}
Since the integrands $c\D a \cdot \Db b$, for $a,b,c\in \dot{H}^1$ may not be absolutely integrable, their integrals must be understood in a limiting sense:  approximate $c$ in the $\dot H^1$-norm by $c_n\in C^\infty\cap \dot H^1(\Sigma)$ and use Wente's estimate to define $\int c\D a \cdot \Db b=\lim_{n\to \infty} \int c_n\D a \cdot \Db b$. We will continue to do this in the sequel without drawing further attention to it.  
\end{remark}

We hence immediately obtain the following well known expressions of the first and second variation of $E$: Given any 
 $u, v,w\in \dot H^1(\Sigma, \R^3)$ we have 
\beq \label{eq:dE}
\ed E (u) [v]= \int_{\Sigma}  \D u \cdot \D v - v\cdot \D u \dot \wedge \Db u 
\eeq
while 
$\ed^2 E(u)[w,v]= \ppl{}{t}{s}\vlinesub{s,t=0} E(u+sv+tw)$ is given by  
\begin{equation} \label{eq:d2E}
\ed^2 E(u)[w,v]=\int_{\Sigma} \D w \cdot \D v - 2 u\cdot (\D v \dot\wedge \Db w)=\int_{\Sigma} \D w \cdot \D v - 2(\Db w \wedge u)\cdot \D v  
\end{equation} where the second equality follows as  $u\cdot (\D v \dot\wedge \Db w)=(\Db w \wedge u) \cdot \D v $.
Wente's inequality   
immediately implies that, for all $u_1,u_2, v,w\in \dot{H}^1$, 
\beq\label{est:D2E-H}
	|\ed^2E(u_1)[w,v] - \ed^2 E(u_2)[w,v]| =\left| \int_\Sigma 2(u_1-u_2)\cdot \D v\dot\wedge \Db w\right|  
	\leq  C \|u_1-u_2\|_{\dot H^1}\|v\|_{\dot H^1}\|w\|_{\dot H^1} %
\eeq
for a universal constant $C$ that could be made explicit using the results we recall in Section \ref{s:Wente}.

Furthermore the formula for $\ed^2 E(u)$ combined with \eqref{eq:symm} easily implies that for every map $u$ for which $\na u\in L^p$ for some $p>2$, so in particular for our adapted bubbles $z\in\ZZ$, 
the Jacobi operator $L_u$ 
of the $H$-surface energy
can be written in the form 
 $L_u =\text{Id}+K_u$ for a compact self adjoint operator $K_u$, see Section \ref{subsec:definite}. The analogue statement hence also holds true for the projected Jacobi operator $\hat L_z:=P^{\VV_z}\circ L_z$, $z\in\ZZ$, so in order to establish that Assumption \ref{ass:A} is satisfied in the present setting it suffices to prove

\begin{lemma}\label{lem:nondeg}
Let $(\Sigma, g)$ be a closed oriented Riemannian surface of positive genus with Gauss curvature  $K_g \in \{-1,0\}$ and let $\ZZ$ be the set of adapted bubbles defined in Section \ref{sec:bubble-space}. 

There exists $c_0>0$ and $\la_0 >1$ so that for every $z\in \mathcal{Z}_{\la_0}$ the spectrum of the projected Jacobi operator $\hat L_z:=P^{\VV_z}\circ L_z$ satisfies 
	$$\sigma(\hat L_z)\cap (-c_0,c_0)=\emptyset.$$
\end{lemma}
 
 We will prove this lemma in Section \ref{subsec:definite} by using a careful compactness argument to reduce it to the corresponding result for maps from the sphere which was established in \cite{CM05}.
 
The second key ingredient needed to prove our main results is the expansion of the energy at elements of $\ZZ$ alluded to in the introduction. Due to rotational invariance it suffices to consider elements of the form $z=z_{\la,a}$. 

\begin{lemma} \label{lemma:principal-term}
Let $(\Sigma, g)$ be a closed oriented Riemannian surface of positive genus with Gauss curvature $K_g \in \{-1,0\}$ and let $\ZZ$ be the set of adapted bubbles defined in Section \ref{sec:bubble-space}.
Then for large $\la$ the energy of elements $z_{\la,a}$ is given by 
\beq
\label{eq:energy-expansion} 
	E(z_{\lam,a}) = \frac{4\pi}{3} - \frac{4\pi}{\lam^2} \mathcal{J}(a)
	+ O(\lam^{-3}), 
\eeq
and its derivative with respect to the bubble scale satisfies 
\beqs
 \de_\la E(z_{\la,a}) =  \frac{8\pi}{\lam^3} \mathcal{J}(a) + O(\lam^{-4}) 
\eeqs
where we use our local coordinates from Remark \ref{rmk:def-Fa} to define
\beq \label{eq:Jdef}
\mathcal{J}(a) :=\lim_{x\to 0} (\partial_{y_1}\partial_{x_1}+\partial_{y_2}\partial_{x_2}) G_a(x,0).
\eeq
\end{lemma}
\begin{remark}\label{r:energy-exp}
	The expansions in Lemma \ref{lemma:principal-term} are 
	in analogy with \cite[Proposition 5.1]{CM05} (see also \cite{T00}), where the regular part of the Green's 
	function of a planar domain appears. 
\end{remark}

As discussed in the introduction, and proven in Section \ref{subsec:Bergman}, we have that $\mathcal{J}(a) \leq \mathcal{J}_{(\Sigma,g)}$ for all $a\in \Sigma$ and some uniform constant $\mathcal{J}_{(\Sigma,g)}<0$.  

We will hence want to apply our abstract result for $y_{z_{\la,a}} = \frac{\de_\la z_{\la,a}}{\|\de_\la z_{\la,a}\|_{\dot H^1}}$ and will use that 
\beq\label{est:H-norm-dla}
\norm{\partial_\la z_{\la,a}}_{\dot H^1}\simeq \la^{-1}\quad  \text{ while } \quad \norm{\partial_\la z_{\la,a}}_{L^2}\leq C \la^{-2}(\log \la)^\half
\eeq
as a short calculation, carried out in Section \ref{subsec:bubble-prop}, shows. 
Here and in the following we use the expression $a\simeq b + O(f(\la))$ to mean that there exists $C>0$ so that $$O(f(\la)) + C^{-1}b \leq a \leq Cb + O(f(\la)).$$   

In addition, we shall prove the following control on the first and second variations of $E$. 
\begin{lemma}\label{lemma:key-estimates}
 Let $(\Sigma, g)$ be a closed oriented Riemannian surface of positive genus with Gauss curvature  $K_g \in \{-1,0\}$ and let $\ZZ$ be the set of adapted bubbles defined in Section \ref{sec:bubble-space}.
Then there exist constants $C\in \R$ and $\la_2>1$ so that for all $z_{\la,a}\in \ZZ_{\la_2}$
\beqs
\norm{\ed^2E(z_{\la,a})[\partial_\la z_{\la,a},\cdot]}_{H^{-1}}\leq C\la^{-3}(\log \la)^\half
\eeqs
and 
\beq
\label{est:H-f2}
\norm{\ed E(z_{\la,a})}_{H^{-1}}\leq C\la^{-2}(\log \la)^\half.
\eeq
\end{lemma}

Combined, these lemmas imply that \eqref{ass:f0}-\eqref{ass:nuu} of
 Assumption \ref{ass:B} are satisfied
for 
\beq
\label{eq:fo-for-H}
y_{z_{\la,a}} = \tfrac{\de_\la z_{\la,a}}{\|\de_\la z_{\la,a}\|_{\dot H^1}}\simeq \la \partial_\la z_{\la,a}, \,\,\, \nu=2,  \,\,\, f_0(\la) =c_1 \la^{-2} \, \text{ and }  \,f_1(\la)=f_2(\la)=C\la^{-2}\log (\la)^\fr12
\eeq
and hence that all of the conditions of Assumption \ref{ass:B}
are satisfied as 
$$\frac{f_1(\la)f_2(\la) + f_2(\la)^2}{f_0(\la)} = C \la^{-2}\log \la\to 0 \text{ as } \la\to \infty.$$
We note that \eqref{est:D2E-H} also implies that  Assumption \ref{ass:C} holds as we note that the norms of $z\in \mathcal{Z}_{\la_0}$ are uniformly bounded for any $\la_0 > 1$. 

This is sufficient to apply Theorem \ref{thm:abstract-critical-points} which combined with Remark \ref{rmk:A2}  immediately implies the energy gap for critical points of the $H$-energy which we claimed in Theorem \ref{t:gap}.

To obtain the claimed \Loj-type estimates we then use that the first claim of Lemma \ref{lemma:principal-term} implies that Assumption \ref{ass:E} holds for $f_4(\la)=C\la^{-2}$.

If we wish to prove a \Loj-type estimates involving the $L^2$ norm of the gradient of $E$ and hence use $\norm{\cdot}_{*}=\norm{\cdot}_{L^2}$ then \eqref{est:H-norm-dla} implies that Assumption \ref{ass:D} holds for $f_3(\la)=\la^{-1}(\log \la)^\half$, while for the proof of Theorem \ref{thm:Loj-H-1} we will simply set $\norm{\cdot}_*=\norm{\cdot}_{\dot H^1}$ resulting in $f_3=1$.

In both cases the assumptions of Theorem \ref{thm:abstract-critical-points} are hence satisfied for $y_z:=\frac{\partial_\la z}{\norm{\partial_\la z}_{\dot H^1}}$ and 
$$(\gamma_0,\si_0)=(\gamma_4,\si_4)=(2,0) \text{ and } (\gamma_1,\si_1)=(\gamma_2,\si_2)=(2,-\thalf)\text{ and } \nu=2$$
as well as $(\gamma_3,\si_3)=(1,-\half)$, and thus $(\gamma_{1,3},\si_{1,3})=(1,-\half)$ if 
 $\norm{\cdot}_*:=\norm{\cdot}_{L^2}$, while $(\gamma_{1,3},\si_{1,3})=(\gamma_3,\si_3)=(0,0)$ if  $\norm{\cdot}_*:=\norm{\cdot}_{\dot H^1}$.

We hence obtain $L^2$-\Loj-inequalities with exponents, compare 
Remark \ref{rmk:coeff-thm-abstract},
\beqas
(\tfrac{1}{\al_1},\tfrac{\beta_1}{\al_1})&=\FF_\infty((1,\thalf),(1,0))=(1,\thalf),\\
(\al_2,\beta_2)&=\FF_0((1,0),(2,\tfrac32))=(1,0), \\
(\al_3,\beta_3) &=\FF_0((2,0),(2,1))=(2,1),
\eeqas
as claimed in Theorem \ref{thm:main} and  $H^{-1}$-\Loj-inequalities with exponents 
\beqas
(\tfrac{1}{\al_1},\tfrac{\beta_1}{\al_1})&=\FF_\infty((2,0),(1,0))=(2,0),\\
(\al_2,\beta_2)&=\FF_0((1,0),(1,\thalf))=(1,\thalf), \\
(\al_3,\beta_3)&=\FF_0((2,1),(1,0))=(1,0),\eeqas
as claimed in Theorems \ref{thm:Loj-H-1}. Note that, using the corresponding scaling of $\|\de_\la z_\la \|_{L^p}$ for $2<p<\infty$, one can also obtain \Loj-estimates involving the $L^q$-norm of $\ed E(u)$ for $1<q<2$.  

We finally remark that while Theorem \ref{thm:abstract-main} is only stated for maps $u$ with $\norm{\ed E(u)}_{*}\leq \half$, the claims of Theorem \ref{thm:main} and \ref{thm:Loj-H-1} are trivially true if $\norm{\ed E(u)}_{*}> \half$. 

\subsection{Illustration of proof of Theorem \ref{thm:abstract-main} in the specific setting of Theorem \ref{thm:main}}
\label{sec:ill}
In this setting the fact that $\norm{\ed^2E(v)-\ed^2E(w)}\leq C\norm{v-w}_{\dot H^1}$ for all $v,w\in \dot H^1$ means that, arguing as in the proof of Lemma \ref{lemma:abstract1}, we get
$$c_0\norm{w}_{\dot H^1}^2\leq C\abs{\ed E(z)[w]}+\abs{\ed E(u)[w]}+C\norm{w}_{\dot H^1}^3,$$
so by \eqref{est:H-f2}
we find that if $\norm{w}_{\dot H^1}$ is sufficiently small then
\beq
\label{est:H-w}
\norm{w}_{\dot H^1}\leq C\la^{-2}(\log \la)^{\half}+C\norm{\ed E(u)}_{H^{-1}}.\eeq
As in the proof of Lemma \ref{lemma:abstract2} we then use \eqref{eq:fo-for-H} to bound
\beqa
c_1\la^{-3}&\leq \abs{\ed E(z)[\plaz]}\leq \abs{\ed E(u)[\plaz]}+\abs{d^2E(z)[\plaz,w]}+C\norm{w}^2\norm{\plaz}_{\dot H^1}\\
&\leq \abs{\ed E(u)[\plaz]}+C\la^{-2}(\log \la)^{\half}\norm{w}_{\dot H^1}+C\la^{-1}\norm{w}_{\dot H^1}^2.
\eeqa
We hence conclude that either 
$\la^{-1}\leq C(\log \la)^{\half} \norm{w}_{\dot H^1}$ 
and thus by 
\eqref{est:H-w} 
$$\la^{-1}(\log \la)^{-\half}\leq C \norm{\ed E(u)}_{H^{-1}}\leq C \norm{\ed E(u)}_{L^2}$$
or that $\la^{-3}\leq C\abs{\ed E(u)[\plaz]}$, which by  
\eqref{est:H-norm-dla} implies that also in this case
\beq \la^{-1}(\log \la)^{-\half}\leq C \norm{\ed E(u)}_{L^2}.\label{est:la-in-extra-proof}
\eeq
If $\la\leq \norm{\ed E(u)}_{L^2}^{-1}$ and thus $(\log \la)^{\half}\leq \abs{\log\norm{\ed E(u)}_{L^2}}^\half$, this immediately gives the claimed bound \eqref{est:mainloj1} on the bubble scale, while \eqref{est:mainloj1} is trivially true if  $\la\geq \norm{\ed E(u)}_{L^2}^{-1}$.

We also see that the first term of the right hand side of \eqref{est:H-w} is lower order, scaling like $\norm{\ed E(u)}_{L^2}^2\abs{\log\norm{\ed E(u)}_{L^2}}^{3/2}$, compared with the second term and thus that we get that $\norm{w}_{\dot H^1}\leq C\norm{\ed E(u)}_{L^2}$ as claimed in \eqref{est:mainloj2}. 

Using first the energy expansion 
\eqref{eq:energy-expansion}, then that $\ed^2 E(\cdot)$ is uniformly bounded on bounded sets of $\dot H^1$ and finally \eqref{est:H-f2} we furthermore obtain
\beqa
\babs{E(u)-\tfrac{4\pi}{3}}&\leq C\la^{-2}+\babs{E(u)-E(z)}\leq C\la^{-2}+\abs{\ed E(z)[w]}+C\norm{w}_{\dot H^1}^2\\
&\leq C\la^{-2} +C\la^{-2}(\log \la)^\half \norm{w}_{\dot H^1}+C\norm{w}_{\dot H^1}^2
\eeqa
and combining this with the just obtained bounds on $\la$ and on $w$ yields the final claim that
$$\left|E(u)-\tfrac{4\pi}{3}\right|\leq C\norm{\ed E(u)}_{L^2}^2(1+\abs{\log\norm{\ed E(u)}_{L^2}}).$$


\section{Relation between $H$-surfaces and Wente estimates}\label{s:Wente}
The $H$-functional is closely related to so-called Wente estimates (see \eqref{eq:West}), we therefore spend some time elucidating the connections between these two perspectives. 

It is well known (\cite[Theorem 5.2]{G98}) that given any $a,b\in \dot H^1(\Sigma)\sm\{0\}$ there is a unique $\phi=\phi_{ab} \in \dot H^1(\Si)$ solving 
\begin{equation}\label{eq:phiab}
-\Dl_g \phi = \{a,b\}_g :=  \D a\cdot_g \Db b, 	
\end{equation}
 where $\{a,b\}_g$ is the Jacobian determinant and $\Db b = (\ast_g \ed b)^\sharp$. Notice that in oriented isothermal coordinates  the metric $g$ can be written in the form $g = e^{2\p} ((\ed x^1)^2 + (\ed x^2)^2)$, $\Delta_g = e^{-2\p} \left( \partial^2_{x_1} + \partial^2_{x_2}  \right)$, 
and $\D a\cdot_g \Db b = e^{-2\p} (a_{x_1} b_{x_2} - a_{x_2} b_{x_1})$.
 
One can check this by, for example, an approximation argument coupled with the following celebrated \emph{a-priori} estimate, originally due to Wente \cite{We69}, which says that the solution to \eqref{eq:phiab} satisfies 
\begin{equation}\label{eq:West}
\|\phi_{ab}\|_{L^\infty} + \|\phi_{ab}\|_{\dot H^1} \leq C((\Sigma,g))\|a\|_{\dot H^1}\| b\|_{\dot H^1}. 	
\end{equation}
In particular for $a,b \in \dot H^1(\Sigma)$ and $c\in C^\infty\cap \dot H^1(\Si)$ 
we can estimate
\beq \label{est:Wente}
\left|\int_\Sigma c\D a \cdot \Db b\right| = \left|\int_\Sigma c \Dl \phi_{ab}\right| =\left|\int_\Sigma \D c \cdot \D \phi_{ab}\right| \leq \|c\|_{\dot H^1}\|\phi_{ab}\|_{\dot H^1} \leq C_W\|a\|_{\dot H^1}\|b\|_{\dot H^1}\|c\|_{\dot H^1}.
\eeq
In order to show that the optimal $H^1$-Wente constant is $C_W =\sqrt{\frac{3}{32\pi}}$, Ge \cite{G98} considered the \textit{Wente energy} 
$W:\dot H^1(\Si)\setminus \{0\}\times \dot H^1(\Si)\setminus \{0\} \to \R\cup\{\infty\}$ which is defined as 
\begin{equation*}
W(a,b) : =  \frac{\|a\|_{\dot H^1}^2\|b\|_{\dot H^1}^2}{\|\phi_{ab}\|_{\dot H^1}^{2}}. 
\end{equation*}
Ge proved that $\inf_{a,b\in \dot H^1(\Sigma)\setminus\{0\}} W(a,b) = \frac{32\pi}{3}$ for every closed Riemann surface $\Sigma$, with $W(a,b) = \frac{32\pi}{3}$ if and only if $\Sigma \cong \Sp^2$ and $u=u(a,b)$ defined in \eqref{eq:WE} is a degree-one conformal map $u:\Sp^2\to\Sp^2\In \R^3$. Since there is a close relationship between the $H$-energy and the Wente energy, our main results allow us to obtain quantitative control on the behaviour of functions $(a,b)$ with near-optimal Wente energy, see Corollary \ref{cor:Wente} and the proceeding remark.   

To this end we first note that since the Wente energy remains invariant under rescaling of each component it suffices to consider functions 
in the unit sphere  $S_1(\dot H^1)$ of $\dot H^1(\Si)$ for which we have the following relationship between $E$ and $W$.   
\begin{lemma}\label{lem:WE}
Let $(\Sigma, g)$ be a closed oriented Riemannian surface of positive genus with Gauss curvature  $K_g \in \{-1,0\}$. Given any $(a,b)\in S_1(\dot H^1)\times S_1(\dot H^1)$ with finite Wente energy we set 
\begin{equation}\label{eq:WE}
u=\frac{W(a,b)^{\fr12}}{2}\left(a,b, \frac{\phi_{ab}}{\norm{\phi_{a,b}}_{\dot H^1}}\right)
 \in \dot H^1(\Sigma,\R^3).
\end{equation}
Then the $H$-surface energy $E$ and the Wente energy $W$, and their variations, are related by 
\beqs
W(a,b) =8 E(u) \text{ and  } \ed E(u)[w]=\tfrac14 W(a,b)^{-\fr12}\ed W(a,b)[w_1,w_2]\eeqs
for any $w=(w_1,w_2,w_3)\in \dot H^1(\Si,\R^3)$. 
In particular $u$ defined by \eqref{eq:WE} is a critical point of $E$ if and only if $(a,b)$ is a critical point of $W$. 
\end{lemma}
\begin{proof} 
As $\phi_{ab}$ is a solution of \eqref{eq:phiab} 
we obtain that $\ed E(u)[(0,0,w^3)]=0$ for all $w^3\in \dot{H}^1$ directly from \eqref{eq:dE}. Due to the symmetry of $W$ in $a$ and $b$ it hence suffices to check the formula for variations of the form $w=(w_1,0,0)$ and we compute, using \eqref{eq:dE} in the last step

\beqas
\ed  W(a,b)[(w^1,0)] &=  2W\int_\Sigma \na a\cdot\na w^1-2W^2\int \phi_{ab}\na  w^1\cdot \Db b  
\\ &=
4W^{\fr12}\int_\Sigma \left[  \D u^1\cdot \D w^1 - 2w^1\D u^2\cdot \Db u^3 \right]
= 4W^{\fr12}\ed E (u) [(w^1,0,0)].
\eeqas
We note that this relation between the variations of $E$ and $W$ combined with the scaling invariance of $W$ immediately implies that 
$\ed E(u)[u]=0$ for $u$ as in \eqref{eq:WE}.
Recalling that $\|u^i\|_{\dot H^1}=\|u^j\|_{\dot H^1} \neq 0$, and in conjunction with \eqref{eq:dE} we thus get 
$$E(u) = \tfrac16 \|u\|_{\dot H^1}^2 +\tfrac13 \ed E(u)[u] =\tfrac12 \|u^1\|_{\dot H^1}^2 +\tfrac13 \ed E(u)[u] = \tfrac12 \|u^1\|_{\dot H^1}^2 =\tfrac18 W(a,b).$$
\end{proof}

This relationship immediately allows us to conclude the following from Theorem \ref{thm:main} and Remark \ref{rmk:althyp}.  

\begin{corollary}\label{cor:Wente}
Let $(\Sigma, g)$ be a closed Riemannian surface of genus larger than zero with Gauss curvature  $K_g \in \{-1,0\}$. 
Then there exist $\eps>0$ and $C<\infty$ so that the following holds: 
	Suppose $(a,b)\in S_1(\dot H^1)\times S_1(\dot H^1)$ satisfies  
	$$\left|W(a,b) - \frac{32\pi}{3} \right| <\eps.$$
	Then $u=u(a,b)$ defined by \eqref{eq:WE} is close to the set of adapted bubbles $\ZZ$ defined in Section \ref{sec:bubble-space}
\beqs 
{\rm{dist}}_{\dot H^1}(u,\mathcal{Z}) \leq C \|\ed W(a,b)\|_{L^2},
\eeqs
the bubble scale of $z\in \mathcal{Z}$ with ${\rm{dist}}_{\dot H^1}(u,\mathcal{Z}) = \|u-z\|_{\dot H^1}$ is bounded by 
	\beqs 
\la(z)^{-1}\leq 
C (1+|\log \|\ed W(a,b)\|_{L^2}|^\fr12)\|\ed W(a,b)\|_{L^2}
	\eeqs 
and we have
\beqs 
\left|W(a,b) - \frac{32\pi}{3} \right|\leq C(1+|\log \|\ed W(a,b)\|_{L^2}|) \|\ed W(a,b)\|_{L^2}^2. 
\eeqs
\end{corollary}
\begin{remark}
Additionally we also obtain the following estimates in terms of $\|\ed W(a,b)\|_{H^{-1}}$ 
\beqas
\la^{-1}&\leq C\norm{\ed W(a,b)}_{H^{-1}}^\half ,\\
\dist_{\dot H^1}(u,\mathcal{Z})&\leq C\norm{\ed W(a,b)}_{H^{-1}}\abs{\log\norm{\ed W(a,b)}_{H^{-1}}}^{\half}, \\ 
\left|W(a,b) - \tfrac{32\pi}{3} \right|&\leq C\norm{\ed W(a,b)}_{H^{-1}}.
\eeqas
\end{remark}
%


\section{Analysis of the $H$-energy on $\ZZ$} 
\label{sec:proofs-of-lemmas}

In this section we give the proofs of the results that we stated in Section \ref{sec:main-proofs} and there used to prove our main results on the $H$-energy.

\subsection{Properties of the adapted bubbles} \label{subsec:bubble-prop}
We first collect some useful properties of the adapted bubbles $z_{\la,a}=\tilde z_{\la,a}-\dashint_{\Sigma} \tilde z_{\la,a}$ defined in Section \ref{sec:bubble-space} which we shall later use in the proofs of Lemmas \ref{lem:nondeg}-\ref{lemma:key-estimates}.
For this we will repeatedly work in local coordinates $x=F_a(p)$ on balls $B_\iota(a)$ as described in Remark \ref{rmk:def-Fa} where we will always assume that the same isometry $F_a$ is used in the choice of coordinates and in the definition of $z_{\la,a}$. 

Writing for short $j_{\lambda,a}(x): = \left(\frac{2}{\lambda}\D_y J_a (x,0),0\right)$ and $J_a$ the regular part of the Green's function, we recall
from \eqref{eq:zhat} that $\tilde z_{a,\la}$ is described in these local coordinates on 
 $B_\iota(a)$ by
\beqa
\hat{z}_{\lambda,a} (x)&= (\bl_{\lambda }(x) + (0,0,1)) + j_{\lambda,a}(x)  \\
& \quad +(\phi(x)-1)\left(\left(-\tfrac{2}{\lambda} \D_y G_a(x,0)+\tfrac{2}{\lambda}\D_y J_a(x,0), 0\right) + (\bl_{\lambda }(x) + (0,0,1)) \right) \\
&= \bl_{\lambda }(x) + (0,0,1) + j_{\lambda,a}(x) + O(\lambda^{-2}).\label{eq:locexp}
\eeqa
We stress that the above expansion continues to hold for all spatial derivatives, and uniformly for $x\in \mathbb{D}_{\rho}$, while the error term in the analogue expression for $\partial_\la \hat{z}_{\lambda,a}$ is of order $O(\la^{-3})$. 

As the first two components of $j_{\la,a}$ are given by derivatives of the regular part of the Green's function, and are hence harmonic, we furthermore have that
\beqs 
\Dl_x 	\hat z_{\lambda, a} (x)=\Dl_x \bl_\la (x)+O(\la^{-2})=\la^2\Dl_x \bl (\la x)+O(\la^{-2}). \label{eq:dlz}\eeqs
We also notice that
\beqs 
\dashint_{\Sigma} \tilde{z}_{\la,a} = O(\la^{-1})+ \Vol_g(\Si)^{-1} \int_{\mathbb{D}_\rho} (\pi_\la(x) + (0,0,1)) \sqrt{g}  \id x=O(\la^{-1}),\eeqs
since $\sqrt{g}\equiv 1$ respectively $\sqrt{g}= 4(1-\abs{x}^2)^{-2}$ are radially symmetric and hence the 
 first two components of $\pi_\la(x)+(0,0,1)=(\frac{2\la x}{1+\la^2\abs{x}^2},\frac{2}{1+\la^2\abs{x}^2})$ have zero average on such discs.

As our adapted bubbles are simply given in terms of the gradient of the Green's function outside of the ball $ F_a^{-1}(\DD_{2r})\subset B_\iota(a)$ we have that
\beq \label{est:z-away-from-ball}
z_{\la,a}=O(\la^{-1}), \quad \D z_{\lambda, a}=O(\la^{-1})\text{ and } \Dl 	z_{\lambda, a}\equiv 0 \text{ on }\Si\setminus F_a^{-1}(\DD_{2r}).
\eeq
Also, a short calculation shows that on the annulus where we interpolate we have
\beq \label{est:z-on-annulus}
z_{\la,a}=O(\la^{-1}), \,\, \D z_{\lambda, a}=O(\la^{-1}) \text{ and }  \Dl z_{\lambda, a}\equiv O(\la^{-2})\, \text{on $ F_a^{-1}(\DD_{2r}\setminus \DD_{r})$} .
\eeq
We note that the derivative with respect to $\la$ of each term appearing above is analogously bounded by $O(\la^{-(s+1)})$ where $s$ is the order of decay appearing in \eqref{est:z-away-from-ball} and \eqref{est:z-on-annulus}. To prove the claims on the behaviour of the first and second variations of the energy at adapted bubbles as well as the estimates on the norms of $\partial_\la z_{\la,a}$ claimed in Section \ref{sec:main-proofs} we will combine these expansions with properties of the stereographic projection $\pi$. It is hence useful to recall that 
 \begin{equation*}
 \D \bl = \frac{2}{(1+|x|^2)^2} \left( \begin{array}{cc}
1-x_1^2 + x_2^2 & \,\,\,\,\,\,\,-2x_1x_2  \\
 -2x_1x_2&  \,\,\,\,\,\,\, 1+x_1^2-x_2^2 \\ 
-2x_1 & \,\,\,\,\,\,\, -2x_2
\end{array} \right), \ x\cdot \D\bl =\frac{2}{(1+|x|^2)^2} \left((1-|x|^2)x, -2|x|^2\right)
 \end{equation*}
so 
in particular 
$0<\|\D (x\cdot \D \bl)\|_{L^2(\R^2)}<\infty$ and $|\de_\la \bl_\la |^2=|x\cdot \D \bl (\la x)|^2=\frac{4|x|^2}{(1+\la^2|x|^2)^2}$. 

Combining these expressions with \eqref{eq:locexp}-\eqref{est:z-on-annulus} and the proceeding comment we obtain 
\beqa 
\|\de_\la z_{\la,a}\|_{\dot H^1}^2 &\simeq \int_{\mathbb{D}_{r}} |\D \de_\la \bl_\la|^2 + O(\la^{-4})\nn =\la^{-2}\int_{\mathbb{D}_{r}} |\D (\la x \cdot \D \bl (\la x))|^2 + O(\la^{-4})\nn\\
&= \la^{-2}\int_{\mathbb{D}_{r\la}} |\D (x\cdot \D \bl)|^2  +O(\la^{-4})\simeq \la^{-2} \nn 
\label{eq:yz}  
\eeqa
as claimed in \eqref{est:H-norm-dla}. Here and in the following we use the scaling properties of the norm $\|\D \cdot\|_{L^2}$ and the obvious estimate $|a+b|^2 \geq \fr12 |a|^2 - |b|^2$ in order to get an appropriate lower bound.


Similarly, \eqref{eq:locexp}-\eqref{est:z-on-annulus} with the proceeding comment imply that
\beqas
\norm{\de_\la z_{\la,a}}_{L^2(\Si,g)}^2 &\simeq\int_{\mathbb{D}_{r}} |\de_\la \bl_\la |^2 + O(\la^{-4})
\simeq\la^{-4}\int_{\mathbb{D}_{r\la}}\frac{|x|^2}{(1+ |x|^2)^2} + O(\la^{-4}) = O(\la^{-4}\log \la). 
\eeqas
In the proof of Lemma \ref{lem:nondeg} we will furthermore use

\begin{lemma}\label{lemma:proj}
Let $(\Sigma, g)$ be a closed oriented Riemannian surface of positive genus with Gauss curvature  $K_g \in \{-1,0\}$ and let $\{z_k=z_{\la_k,a_k}\}$ be a sequence of adapted bubbles for which $\la_k\to \infty$ and let $\{v_k\}$ be a bounded sequence in $ \dot H^1(\Si)$ which is so that $v_k(F_{a_k}^{-1}(\la_k^{-1}\cdot))$ converges locally in $ H^1(\R^2)$ to some limit $v_\infty$. 
Then also the $\dot H^1$-orthogonal projections $P^{T_{z_k}\ZZ}(v_k)$ onto the tangent space of the adapted bubble set converge to the projection of the limit onto the tangent space of the bubble set $\mathcal{B}^1$ defined in \eqref{def:B1} 
 in the sense that 
$$\big(P^{T_{z_k}\ZZ}v_k\big)(F_{a_k}^{-1}(\la_k^{-1}\cdot))\to P^{T_\pi \mathcal{B}^1} v_\infty \text{ smoothly locally on } \R^2$$
while 
$$\lim_{\Lambda\to \infty}\lim_{k\to \infty} \|\na P^{T_{z_k}\ZZ}v_k\|_{L^2(\Sigma\sm B_{\Lambda\la_k^{-1}}(a))}=0.$$
\end{lemma}

The proof of this lemma follows as the above expressions of the adapted bubbles allow us to show that for any $b\in \R^2$, any $\om\in so(3)$ and any $c\in \R$ the renormalised variations $\frac{\partial_\eps z_k^\eps}{\norm{\partial_\eps z_k^\eps}_{\dot H^1}}\vert_{\eps=0}$, 
$z_k^\eps=(I+\eps\omega) z_{\la_i(1+c),F_{a_i}(\eps b)}$, converge to the corresponding variation of $\bl$ in the sense described in the above theorem.
We note that in order to define the variations with respect to $a$ we can use the pseudo canonical choice of $F_a$ described in Remark \ref{rmk:coord}.

\subsection{Definiteness of the second variation}
\label{subsec:definite}
Here we give a proof of the  uniform definiteness of the second variation of $E$ in directions orthogonal to the bubble set claimed in Lemma \ref{lem:nondeg} which assures that Assumption \ref{ass:A} holds.

The Jacobi operator $L_u:\dot H^1(\Si,\R^3)\to \dot H^1(\Si,\R^3)$ is characterised by, see \eqref{eq:d2E},
\begin{eqnarray}
\int \na L_u(w)\cdot \na v = \ed^2 E(u)[w,v]= \int_{\Sigma} \D w \cdot \D v - 2(\Db w \wedge u)\cdot \D v \nn \text{ for all } v\in  \dot H^1(\Si,\R^3)
\end{eqnarray}
and can thus be written as $L_u (w) := w + c_u(w)$ where $c_u (w)\in \dot H^1(\Sigma, \R^3)$ is defined by 
\begin{equation*} 
-\Dl_g c_u(w) = 2{\rm div}_g(\Db w\wedge u) = 	2\Db w\dot\wedge \D u.  
\end{equation*}
We notice that Wente's estimate \eqref{eq:West} gives the existence of some $C=C(g)$ so that 
\begin{equation*}
	\|c_u(w)\|_{L^\infty} + \|c_u(w)\|_{\dot H^1} \leq C \|u\|_{\dot H^1}\|w\|_{\dot H^1}
\end{equation*}
and note that if $\na u\in L^{p}$ for some $p>2$ then we also have the usual $W^{2,q}$-estimate 
$$\|c_u(w)\|_{W^{2,q}}\leq C\|w\|_{\dot H^1}, \quad \text{ for }q=\tfrac{2p}{p+2} >1 \text{ and } C=C(g,\|\D u\|_{L^p}).$$
In this case 
 $c_u :\dot H^1(\Si,\R^3) \to \dot H^1(\Si,\R^3)$ is compact and self-adjoint 
 so 
 \begin{equation}\label{eq:Jproj}
 \hat L_z(w)=w+P^{\VV_z}c_u (w)	
 \end{equation}
 admits an orthonormal basis of eigenfunctions  $\{w_k\}$ with associated eigenvalues $\mu_k \to 1$.

Finally we recall the characterisation of the kernel of the Jacobi operator of $E$ at $\bl:\R^2\to \R^3$ given in \cite[Lemma 9.2]{CM05}. We restate this result in different notation and leave it to the reader to check the details (noting that we ignore the constant Jacobi fields since we work in $\dot H^1$).
\begin{lemma}[\cite{CM05}, Lemma 9.2]\label{lem:ker} Let $L_\bl$ be the Jacobi operator of $E$ at $\bl:\R^2\to \R^3$ then
	$$\ker (L_{\bl}) = T_{\bl}\mathcal{B}^1=\text{span}\{\om\bl,  \left(a\cdot \D\bl \right)^{\top_{\dot H^1}} ,\left(x\cdot \D \bl \right)^{\top_{\dot H^1}}, \,\om \in so(3),\,a\in \mathbb{R}^2\}\In \dot H^1(\R^2,\R^3),
	$$
	where $\cdot^{\top_{\dot H^1}}$ denotes the projection onto $\dot{H}^1(\R^2,\R^3)$ and $\mathcal{B}^1$ as defined in \eqref{def:B1}.  
	Notice that the spanning set above is by no means $\dot H^1$-orthogonal. 
\end{lemma}

\begin{proof}[Proof of Lemma \ref{lem:nondeg}]
Since the spectral properties of $z\in \ZZ$ are invariant under ambient rotations (i.e. by $R\in SO(3)$) it suffices to prove that there are no sequences $z_k=z_{\la_k,a_k}$ and $\{w_k\} \In \VV_{z_k}$ with $\norm{w_k}_{\dot H^1}=1$ so that $ \hat{L}_{z_k}w_k = \al_k w_k$ for some $|\al_k|\to 0$.
We prove the result by contradiction: given such sequences $\{z_k\}$, $\{w_k\}$ we set $c_k = c_{z_k}(w_k)$ and note that by 
 \eqref{eq:Jproj}
\begin{equation}\label{eq:rel}
\text{$w_k = -\mu_k P^{\VV_{z_k}} (c_k) = \mu_k (P^{T_{z_k} \ZZ}(c_k) - c_k) $  for  $\mu_k=\tfrac{1}{1-\al_k}\to 1$.}	
\end{equation}
We will show that re-scaled versions of $w_k$ converge to a non-trivial element of the kernel of $L_\bl$, which is simultaneously orthogonal to $T_\bl \mathcal{B}^1$ leading to a contradiction due to Lemma \ref{lem:ker}.

For $x\in \DD_{\la _k\rho}$ we consider the re-scaled functions, 
denoting by $F_k= F_{a_k}$ the isometry used implicitly to give $z_k$ the explicit form above,
\begin{eqnarray*}
\text{$\check{c}_k (x)= c_k(F_k^{-1}(\la_k^{-1}x))$, $\check{w}_k(x) = w_k(F_k^{-1}(\la_k^{-1}x))$ and $\check{z}_k=z_k(F_k^{-1}(\la_k^{-1}x))$.}
\end{eqnarray*}
By construction $\check{z}_k$ converges locally smoothly to $\bl$ on $\R^2$ and 
$\check{c}_k$ solves 
\beq \label{eq:PDE}
-\Dl \check{c}_k =2\D \check{w}_k \dot\wedge \Db \check{z}_k.\eeq
Hence $\{\check{c}_k\}$ has locally uniform bounds in $W^{2,2}_{loc}(\R^2)$ so converges (up to subsequence) strongly in $H^1_{loc}(\R^2)$ to some limit $c\in H^1(\R^2)$. Due to Lemma \ref{lemma:proj} and \eqref{eq:rel} this implies that also $\check{w}_k \to w\in H^1(\R^2)$ locally strongly in $H^1$. Furthermore by passing to the limit in \eqref{eq:PDE}, we get 
\begin{equation*}
\text{$-\Dl c = 2\D \bl \dot\wedge \Db w$, so in fact $c=c_\pi (w)$.}	
\end{equation*}
Once again utilising Lemma \ref{lemma:proj} to pass to the limit in \eqref{eq:rel}, we get $w=-P^{\VV_\bl}(c_\bl (w))$, so $w\in \VV_\pi$ and $\hat{L}_\bl (w)=w+P^{\VV_\bl}(c_\bl (w)) = 0$. Since $L_\bl y = 0$ for all $y\in T_\bl \mathcal{B}^1$, and $L_\bl$ is self-adjoint, we can now conclude that $L_\bl w = 0$. Thus by Lemma \ref{lem:ker} we must have $w\equiv 0$. This will lead to a contradiction as the strong convergence of $\check{w}_k$ to $w$ on any $\DD_\Lambda$ implies that 
 $$\|\D w\|^2_{L^2(\R^2)}=\lim_{\Lambda \to \infty}\lim_{k\to \infty} \|\D\check{w}_k\|^2_{L^2(\DD_\Lambda)} = 1- \lim_{\Lambda \to \infty}\lim_{k\to \infty} \|\D w_k\|^2_{L^2(\Sigma \sm B_{\lambda_k^{-1}\Lambda}(a_k))} = 1$$ where the last equality will follow from \eqref{eq:rel} and Lemma \ref{lemma:proj}
 once we show  that
   \begin{equation}\label{eq:claim}
	\lim_{\Lambda \to \infty} \lim_{k\to \infty}\|\D c_k\|_{L^2(\Sigma \sm B_{\Lambda \lambda_k^{-1}} (a_k))}=0.
\end{equation}
 \textbf{Proof of \eqref{eq:claim}:} 
We first note that $c_k \to 0$ strongly in $H^2_{loc}(\Sigma \sm\{a\})$: this follows from the fact that (locally away from $a$) $\|\Dl c_k\|_{L^2} = O(\la_k^{-1})\|\D w_k\|_{L^2}=O(\la_k^{-1})$ and that any weak limit of $c_k$ over the whole of $\Sigma$ must have zero mean.  Thus it remains to consider $c_k$ in regions of the form $F_k(B_{\iota}(a_k) \setminus B_{\la_k^{-1}\Lambda}(a_k))\subset \DD_\rho\sm \DD_{\la_k^{-1}\Gamma}$ where $C\Gamma \geq  \Lambda$ and in the following $c_k$ is expressed in the coordinates determined by $F_k$.  
For such domains we claim that 
\begin{equation}\label{eq:annuli}
\text{$\lim_{\Gamma\to \infty}\lim_{k\to \infty} R_\Gamma(c_k)=0$, \quad where \quad $R_\Gamma(c_k)=\sup_{\la_k^{-1}\Gamma/2\leq R \leq \rho/2} \int_{\DD_{2R}\setminus \DD_R} |\D c_k|^2.$}	
\end{equation}
If this were not true then we could find a sequence $R_k \to 0$ so that $R_k\la_k\to \infty$ and $\bar{c}_k(x):=c_k(R_k^{-1} x)$ does not converge locally strongly to a constant in $H_{loc}^1(\R^2\sm \{0\})$. However, on compact subsets $K$ of $\R^2\sm \{0\}$ we have 
$\|\Dl \bar{c}_k\|_{L^2(K)}=O ((R_k\la_k)^{-1})\|\D \bar{w}_k \|_{L^2(K)}=O ((R_k\la_k)^{-1})\to 0.$ In other words $\bar{c}_k$ converges locally strongly on $\R^2\sm\{0\}$ to a harmonic function on $\R^2\sm\{0\}$ with finite Dirichlet energy i.e. a constant. Thus we must have \eqref{eq:annuli}. 

In order to finish the proof we will utilise Lorentz space estimates (see \cite[Section 3.3]{H02} for a short introduction): given \eqref{eq:annuli} it is much easier to prove \eqref{eq:claim} for the weaker quantity $\|\D c_k\|_{L^{2,\infty}}$ rather than $\|\D c_k\|_{L^2}$. Furthermore the special structure of the equation, utilising Hardy Space methods, allows one to prove a global upper bound on the stronger quantity $\|\D c_k\|_{L^{2,1}}$. One can couple these estimates by $L^{2,\infty}$-$L^{2,1}$ duality which yields the desired estimate. Such arguments are now common in the so-called neck analysis of bubble convergence but were initially introduced by Lin-Rivi\`ere (see e.g. \cite{LinR01}).      

We let $d_k\in H_0^1(\DD_\rho\setminus \DD_{\la_k^{-1}\Gamma/2})$ solve $-\Dl d_k =- \Dl c_k =2\D w_k \dot\wedge \D z_k$, so that the usual Wente estimates (see e.g. \cite[Theorem 1.3]{G98}) give 
$$\|\D d_k \|_{L^2(\DD_\rho\setminus  \DD_{\la_k^{-1}\Gamma/2})}\leq C(\|\D \pi\|_{L^2(\R^2\setminus \DD_{\Gamma/2})} + O(\la_k^{-1})).$$
As $h_k: = c_k -d_k$ is harmonic we obtain that for any $R\in [\la_k^{-1}\Gamma,\rho/3]$ and any $x\in \DD_{2R}\sm \DD_{R}$
\begin{eqnarray}
|\D h_k(x)|&\leq&  C\dashint_{\DD_{R/4}(x)} |\D h_k|\leq \frac{C}{|x|} (R_\Gamma(c_k)+\|\D \bl\|_{L^2(\R^2\setminus \DD_{\Gamma/2})} + O(\la_k^{-1})). \nn \\
 \nn\label{eq:l2infty}	
\end{eqnarray}
In particular
$\|\D h_k \|_{L^{2,\infty}(\DD_{\rho/2}\sm \DD_{\la_k^{-1}\Gamma})} \leq C(R_\Gamma(c_k)+\|\D \bl\|_{L^2(\R^2\setminus \DD_{\Gamma/2})} + O(\la_k^{-1}))$ and therefore 
$$\|\D c_k\|_{L^{2,\infty}(\DD_{\rho/2}\sm \DD_{\la_k^{-1}\Gamma})} \leq C(R_\Gamma(c_k)+\|\D \bl\|_{L^2(\R^2\setminus \DD_{\Gamma/2})} + O(\la_k^{-1})).$$ 
The stronger uniform bound of Coiffman-Lyons-Meyers-Semmes \cite{CLMS93} tell us that $\Dl c_k \in \h^1(\DD_\rho)$, the (local) Hardy space, with a uniform bound. Thus we know (using e.g. \cite[Theorem 3.3.8]{H02}) that $\|\D c_k \|_{L^{2,1}(\DD_{\rho/2})}$ is uniformly bounded. Utilising now the duality of $L^{2,\infty}$ and $L^{2,1}$ gives:
$$\|\D c_k\|^2_{L^2(\DD_{\rho/2}\setminus \DD_{\la_k^{-1}\Gamma})}\leq C\|\D c_k\|_{L^{2,\infty}(\cdot)}\|\D c_k \|_{L^{2,1}(\cdot)}  \leq C(R_\Gamma(c_k)+\|\D \bl\|_{L^2(\R^2\setminus \DD_{\Gamma/2})} + O(\la_k^{-1})),$$ 
which, combined with \eqref{eq:annuli} proves the claim \eqref{eq:claim} and thus completes the proof. \end{proof}

\subsection{Properties of the energy at adapted bubbles}\label{subsec:proofs-lemmas}
We finally give the proofs of Lemmas \ref{lemma:principal-term} and \ref{lemma:key-estimates}. As we will consider $a$ to be fixed we will drop the index $a$ in the following calculations.  

We recall that away from $F_a^{-1}(\DD_{r})$ the adapted bubble $z_\la$ and its derivatives are controlled by \eqref{est:z-away-from-ball} and \eqref{est:z-on-annulus}. The formulae \eqref{eq:dE}, \eqref{eq:d2E} for the variations of $E$ yield that for $w\in \dot H^1 (\Si)$
\begin{eqnarray} 
\ed E(z_\la)[w] =\int_{(F_a)^{-1}(\DD_{r})}w\cdot (-\Dl z_\la -\D z_\la\dot\wedge \Db z_\la)+ O(\la^{-2})\norm{w}_{L^1(\Si\setminus F_a^{-1}(\DD_{r}))},\label{eq:expansion-first-var}\\
\ed^2 E(z_\la)[\de_\la z_\la, w]=\int_{(F_a)^{-1}(\DD_{r})}w\cdot \partial_\la (-\Dl z_\la -\D z_\la\dot\wedge \Db z_\la)+ O(\la^{-3})\norm{w}_{L^1(\Si\setminus F_a^{-1}(\DD_{r}))}\label{eq:expansion-second-var}.
\end{eqnarray}
The conformal invariance of these expressions allows us to work in the coordinates $x=F_a(p)$ on $F_a^{-1}(\DD_{r})$ in which $z_\la$ is given by $\hat z_\la=\pi_\la+j_\la$ (up to an additive constant). As $\pi_\la$ itself is a solution of the $H$-surface equation while $j_\la$ is harmonic on this disc we have
 \begin{eqnarray}
	-\Dl \hat z_\la - \D \hat z_\la \dot\wedge \Db \hat z_\la 
	&=&-2\D \bl_\la \dot\wedge \Db j_\la   +O(\la^{-2}) \label{eq:ELz} \\
\de_\la (-\Dl \hat z_\la - \D \hat z_\la\dot\wedge \Db \hat z_\la ) &=& \partial_\la(
-2\D \bl_\la \dot\wedge \Db j_\la )  +O(\la^{-3}) \label{eq:ELz2}. 
\end{eqnarray}

\begin{proof}[Proof of  Lemma \ref{lemma:key-estimates}]
Let $w\in \dot H^1(\Si,\R^3)$ be any element with $\norm{w}_{\dot H^1}= 1$, in the following expressed in the coordinates determined by $F_a$. From \eqref{eq:expansion-first-var} and \eqref{eq:ELz} we get
\beqas 
|\ed E(z_\la)[w]|& = 2\left|\int_{\mathbb{D}_{r}}  w\cdot	\D (\bl_\la +(0,0,1)) \dot\wedge \Db j_\la \right| + O(\la^{-2}) \\
&\leq 2\int_{\mathbb{D}_{r}} |\bl_\la + (0,0,1)| |\D  w \dot\wedge \Db j_\la | + 2\int_{\de \mathbb{D}_{r}} | w||\bl_\la + (0,0,1)||\D j_\la| + O(\la^{-2})\\
&\leq  O(\la^{-1})  \|\bl_\la + (0,0,1)\|_{L^2(\mathbb{D}_{r})} + O(\la^{-2}) = O(\la^{-2} |\log \la|^\fr12)
\eeqas
while a combination of \eqref{eq:expansion-second-var} and \eqref{eq:ELz2} gives
\beqas
|\ed^2 E(z_\la)[\de_\la z_\la, w]| 
&= \left |2 \int_{\mathbb{D}_{r}}  w\cdot \D \de_\la \bl_\la \dot\wedge \Db j_\la +  w\cdot\D (\bl_\la + (0,0,1)) \dot\wedge \Db \de_\la j_\la \right|  + O(\la^{-3})  \\
&\leq  2\int_{\mathbb{D}_{r}} |\de_\la \bl_\la| |\D  w \dot\wedge \Db j_\la | + 2\int_{\mathbb{D}_{r}} |\bl_\la + (0,0,1)|  | \D w \dot\wedge \Db \de_\la j_\la |\\ 
 &  + 2\int_{\de \mathbb{D}_{r}} | w|\big(|\de_\la \bl_\la ||\D j_\la| + |\bl_\la + (0,0,1)||\D \de_\la j_\la|\big) + O(\la^{-3})  \\ 
 &=  O(\la^{-3} |\log\la|^{\fr12}).
\eeqas
\end{proof}

\begin{proof}[Proof of Lemma \ref{lemma:principal-term}]
As  $w=\partial_\la z_\la$ is of order $O(\la^{-2})$ on $\Si\setminus F_a^{-1}(\DD_{r})$, \eqref{eq:locexp}, \eqref{eq:expansion-first-var} and \eqref{eq:ELz} yield
\beqas
\de_\la E(z_\la)
  &=  - 2\int_{\DD_{r}} \left( \de_\la \bl_\la \cdot (\D \bl_\la \dot\wedge \Db j_\la)  + \de_\la j_\la \cdot (\D \bl_\la \dot\wedge \Db j_\la)
 \right) + O(\la^{-4}), 
\\
&=  - 2\int_{\DD_{r}} \de_\la \bl_\la \cdot (\D \bl_\la \dot\wedge \Db j_\la) +O(\la^{-4})
\eeqas
where the last step follows as the $L^1$-norm of the third component of $\na \pi_\la$ decays like  $O(\la^{-1})$, while $j_\la$ and $\partial_\la j_\la$ are smooth functions of order $O(\la^{-1})$ respectively $O(\la^{-2})$ with vanishing third component.
 Integrating by parts, using that $|j_\la||\de_\la \bl_\la||\D \bl_\la|=O(\la^{-4})$ on  $\partial \DD_{r}$, hence gives 
\beqas
	\partial_\la E (z_\la)
&=    2\int_{\DD_{r}(a)}  \Db \de_\la\bl_\la \cdot  (\D\bl_\la\wedge j_\la)+O(\la^{-4})=-2\int_{\DD_{r}(a)}  j_\la  \cdot ( \D\bl_\la\dot \wedge\Db \de_\la\bl_\la)+O(\la^{-4})\\& =  - \int_{\DD_{r}(a)} j_\la \cdot \de_\la (\D\bl_\la\dot\wedge \Db \bl_\la)+O(\la^{-4})\\
& =   - \de_{x_1} j_\la (0)\cdot \int_{\mathbb{D}_{r}} x_1  \de_\la (\D\bl_\la\dot\wedge \Db \bl_\la) - \de_{x_2} j_\la (0)\cdot \int_{\mathbb{D}_{r}} x_2  \de_\la (\D\bl_\la\dot\wedge \Db \bl_\la)  \\
	&\quad -    j_\la(0)\cdot \int_{\mathbb{D}_{r}} \de_\la (\D\bl_\la\dot\wedge \Db \bl_\la) +O \left( \la^{-1} \int_{\mathbb{D}_{r}} |x|^2 \abs{ \de_\la (\D\bl_\la\dot\wedge \Db \bl_\la)} 
	\right)  +O(\la^{-4})
\eeqas
where we note that the error term obtained in the Taylor expansion is also of order $O(\la^{-4})$ as a short calculation shows.

As $\pi:\R^2\to \Sp^2$ is a conformal harmonic map we have $\D\bl\dot\wedge \Db\bl =-\Delta \bl=\abs{\na \pi}^2 \pi$ and thus 
$$
  \D\bl_\la\dot\wedge \Db \bl_\la = \tfrac{8 \la^2}{(1+\la^2 |x|^2)^2}\pi_\la= \tfrac{8 \la^2}{(1+\la^2 |x|^2)^3}(2\la x, 1-\la^2\abs{x}^2)^T.$$
The symmetries of this term, combined with $j_\la^3=0$, hence imply that 
\beqas
	\partial_\la E (z_\la)
	&=  - \de_{x_1} j_\la^1 (0)\cdot \int_{\mathbb{D}_{r}} x_1 \de_\la \big( \tfrac{16 \la^3x_1}{(1+\la^2 |x|^2)^3}\big)- \de_{x_2} j_\la^2 (0)\cdot \int_{\mathbb{D}_{r}} x_2 \de_\la \big( \tfrac{16 \la^3x_2}{(1+\la^2 |x|^2)^3}\big) 
	+O(\la^{-4}).
\eeqas
For $i=1,2$ we have
\beqas
\int_{\mathbb{D}_{r}} x_i \de_\la \big( \tfrac{16 \la^3x_i}{(1+\la^2 |x|^2)^3}\big)=48 \int_{\mathbb{D}_{r}}\tfrac{\la^2x_i^2(1-\la^2\abs{x}^2)}{(1+\la^2 |x|^2)^4} dx=\frac{24}{\la^2} \int_{\mathbb{D}_{\lam r}}\tfrac{|x|^2 \left(1-|x|^2 \right)}{\left(1+|x|^2\right)^4} dx=   - \frac{4\pi}{ \lam^2} +  O(\lam^{-4}) 
\eeqas
 and we recall from \eqref{eq:locG} that away from $x=0$ we can write 
$j_\lam^i(x)=2\la^{-1}\partial_{y_i} J_a(x, 0)=2\la^{-1}(\partial_{y_i} G_a(x, 0)-\partial_{x_i}\log\abs{x}).$  As 
$\log\abs{x}$ is harmonic away from $0$ we have 
$$\partial_{x_1}j_\lam^1(0)+\partial_{x_2}j_\lam^2(0)=2\la^{-1}\lim_{x\to 0} (\partial_{y_1}\partial_{x_1}+\partial_{y_2}\partial_{x_2}) G_a(x,0)=2\la^{-1}\mathcal{J}(a).
$$
and the claim that  $\ed E (z_\lam)[\de_\la z] = \frac{8\pi}{\lam^3} \mathcal{J}(a)+O(\la^{-4})$ follows.\end{proof}


\section{Analysis of principal term using Bergman's kernel}
\label{subsec:Bergman}

In this section we establish our claim that $\mathcal{J}(a):=\lim_{x\to 0}(\partial_{y_1}\partial_{x_1}+\partial_{y_2}\partial_{x_2}) G_a(x,0)<0$
where $G_a$ represents the Green's function in the coordinates centred at $a$, described in Remark \ref{rmk:def-Fa}. 

For flat tori with unit area, by using translation-invariance,  
for any $a\in \Sigma$ we have that 
$
G_a(x,y)= G_a(0,x-y) = :\hat{G}(x-y). 
$
Therefore, for $x \neq y$
$$
\partial_{x_1} \partial_{y_1} G_a(x,y)  + \partial_{x_2} \partial_{y_2} G_a(x,y) =   - \Delta \hat G (x-y) = - 2\pi, 
$$
giving $\mathcal{J}(a) = -2\pi$. 

\

For higher genus surfaces $\Sigma$ we will show that $\mathcal{J}$ is determined by the diagonal  
of {\em Bergman's kernel}: this is done with a rather easy adaptation of a result in \cite{Schiffer46}. 

Let us first recall some basic facts and notation of complex-analytic nature. Denote by $\mathcal{H}_\Sigma$ the bundle of holomorphic 
one-forms $\phi$ (for the complex structure induced by the  metric), locally expressed in complex coordinates as $\ti\phi(z) \ed z$, with $\ti\phi$ holomorphic 
in a domain of the complex plane. We also denote by $\partial_z, \overline{\partial}_z$ the 
holomorphic and anti-holomorphic exterior differentials and recall that, with this notation, 
if $\{z\}$ is a local complex coordinate, if $f$ is a function and if $\alpha,\beta$ are differential forms locally given by $\alpha = \ti\alpha(z) \ed z$, $\beta = \ti\beta(z) \ed \zb$ then 
\begin{equation*}
\db_{z} f = \pl{f}{\zb}\ed \zb ;\qquad \de_{z}f = \pl{f}{z}\ed z;  \qquad  \db_{z} \alpha =  \frac{\partial \ti\alpha}{\partial \zb} \ed \zb \wedge \ed z; 
\qquad {\partial}_z \beta =  \frac{\partial \ti\beta}{\partial z} \ed z \wedge \ed \zb . 
\end{equation*}
The above operators are independent of the choice of complex coordinate $z$, and hence 
are well defined. The usual exterior differential splits as $d = \partial_z + \overline{\partial}_z$, and the following relations hold 
$$
   \partial_z^2 = \overline{\partial}_z^2 = \partial_z \overline{\partial}_z + \overline{\partial}_z \partial_z = 0. 
$$
When taking complex differentials with respect to two variables $z$ and $\zeta$, we will always 
mean taking the tensor products of the corresponding complex-valued forms.

\

We recall the following definition of the Bergman Kernel which {\em projects} all 1-forms of class $L^2(\Sigma)$ 
onto sections of $\mathcal{H}_\Sigma$.

\begin{definition}\label{def:Bergman}
	The {\em Bergman Kernel} $B_\Sigma$ of $\mathcal{H}_\Sigma$ is the unique section of the bundle 
	$\mathcal{H}_\Sigma \otimes \overline{\mathcal{H}}_\Sigma \to \Sigma \times \Sigma$ 
	satisfying  
	\begin{description}
		\item {\bf a)} $B_\Sigma(z,\zeta)   = \overline{B_\Sigma(\zeta,z)};$
		\item {\bf b)} $\overline{\partial}_z B_\Sigma(z,\zeta)   = 0;$
		\item{\bf c)} for every section $s$ of $\mathcal{H}_\Sigma$ one has the \emph{reproducing formula}, integrating in $\zeta$
		\begin{equation*}
		s(z) = i \int_{\Sigma}  s(\zeta) \wedge B_\Sigma(z,\zeta). 
		\end{equation*}
	\end{description}
\end{definition}
If $\{\phi_j\}_j$ is an $L^2$-orthonormal basis for the sections in $\mathcal{H}_\Sigma$, which has complex dimension equal to the genus of $\Sigma$, then the Bergman Kernel is given by 
\begin{equation}\label{eq:Bergman-basis}
  B_\Sigma(z,\zeta) = \sum_j \phi_j(z)  \otimes \overline{\phi_j(\zeta)}.
\end{equation}
See Section 1.4.1 in \cite{Krantz} for a more general discussion of Bergman Kernels. 
We will show
\begin{prop}\label{p:Bergman-Green}  
Let $(\Sigma, g)$ be a closed Riemannian surface let $G$ be the Green's function defined in \eqref{eq:Grn} and $B$ denote the Bergman Kernel. For all $z \neq \zeta$, the following identity holds
	\begin{equation*}
	B_\Sigma(z,\zeta)  =-\frac{1}{\pi} \bar\de_\zeta \de_z G (z,\zeta). 
	\end{equation*}	
\end{prop}

\begin{remark}\label{rmk:Jsign}
Given a point $a\in \Sigma$ we can simply use $z=x^1 + ix^2$, $\zeta = y^1+iy^2$ as local coordinates near $a$ introduced in Remark \ref{rmk:def-Fa}. Letting $\{\phi_j\}$ be an $L^2$-orthonormal basis of sections of $\h_\Sigma$ we can locally write $\phi_j = \ti\phi_j \ed z$. The proposition above coupled with \eqref{eq:Bergman-basis} now imply that 
$$ \mathcal{J}(a) = 4 \lim_{x\to 0} {\rm Re}\left(\ppl{G_a}{\bar\zeta}{z}(x,0)\right) = -4\pi \sum_j |\ti\phi_j(0)|^2,$$
for $\mathcal{J}$ defined in \eqref{eq:Jdef}. When $(\Sigma,g)$ is a flat torus of unit area we have that $|\ed z|_g  =2^{\frac{1}{2}}$, and since the only holomorphic one-forms are of the form $\phi = \ti\phi \id z$ for $\ti\phi\in \C$ constant, in this case we conclude that $\mathcal{J}(a)=-4\pi |\ti \phi|^2 = -2\pi |\phi|^2_g = -2\pi$ as above. On hyperbolic surfaces $(\Sigma,g)$ of constant curvature $K_g\equiv -1$, we have that $|\ed z(0)|_g=2^{-\frac12}$ in our local coordinates. Thus we obtain 
$$\mathcal{J}(a)=-8\pi \sum_{j} |\phi_j(a)|_g^2$$ 
which is strictly negative for all $a$ by Riemann-Roch (see the introduction).\end{remark}
 
\begin{proof}

For $z\neq \zeta$ we define $\mathfrak{G}(z,\zeta) = \bar\de_\zeta \de_z G(z,\zeta)$. By the symmetries of the Green's function $\mathfrak{G}$ satisfies condition $a)$ off the diagonal. Furthermore again when $z\neq \zeta$, using the definition of Green's functions \eqref{eq:Grn} as well as that
$ \Delta_g^\zeta \varphi \, dV_g = - 2 i \partial_\zeta \overline{\partial}_\zeta \varphi$
we have
$$\de_\zeta \mathfrak{G} (z,\zeta) =-\frac{1}{2i} \de_z \Dl_g^\zeta G (z,\zeta) \ed V_g =\frac{1}{i}\de_z \left(\frac{\pi}{Vol_g(\Sigma)}\right) \ed V_g=0.$$
By part $a)$ thus $\bar\de_z \mathfrak{G} (z,\zeta) =0$ off the diagonal.

Given a point $z_0 \in \Sigma$, we can simply use $z=x^1 + ix^2$, $\zeta = y^1+iy^2$ for local coordinates near $z_0$ introduced in Remark \ref{rmk:def-Fa} (for the corresponding constant curvature metric). In these coordinates, whenever $z\neq \zeta$, $G_{z_0}(z,\zeta) = -\log |z-\zeta| + J_{z_0}(z,\zeta)$ 
as in \eqref{eq:locG}. Thus defining $\mathcal{G}(z,\zeta) = \de_z G(z,\zeta)$ we can locally express, for $\mathcal{G}(z,\zeta) = \tilde{\mathcal{G}}(z,\zeta)\ed z$, 
\beq \label{eq:localG}
  \tilde{\mathcal{G}}(z,\zeta)  = \frac{1}{2(\zeta -z)} +  \pl{J_{z_0}}{z}(z,\zeta).
\eeq 
Thus in our local coordinates, when $z\neq \zeta$, we have $\mathfrak{G}(z,\zeta)= \ppl{J_{z_0}}{\bar\zeta}{z}(z,\zeta)\ed z\otimes \ed \bar\zeta$ which now clearly extends smoothly to the diagonal. 

\

It remains to show that $-\frac{1}{\pi}\mathfrak{G}$ satisfies the reproducing formula $c)$. Considering a holomorphic one-form $s\in \mathcal{H}_\Sigma(\zeta)$ we have, for $z\neq \zeta$
$$
 d_\zeta (s \otimes \mathcal{G}(z,\zeta) )  =- s \wedge \bar\de_\zeta \mathcal{G} (z,\zeta)= - s  \wedge 
 \mathfrak{G}(z,\zeta).  
$$
By Stokes' theorem, integrating in $\zeta$ and locally writing $s=\ti{s}(\zeta) \ed \zeta$ for some holomorphic $\ti{s}(\zeta)$ 
\begin{eqnarray*}
\int_{\Sigma} s\wedge  \mathfrak{G}(z_0,\zeta)   &=& \lim_{\eps \to 0} \int_{\Sigma \setminus F_{z_0}^{-1}(\DD_\eps)} s \wedge 
 \mathfrak{G}(z_0,\zeta)   =  
 -  \lim_{\eps \to 0}\int_{\Sigma \setminus F_{z_0}^{-1}(\DD_\eps)}  d_\zeta (s\otimes \mathcal{G}(z_0,\zeta)  )  \\ 
 & = &  \lim_{\eps \to 0}\left( \int_{\partial \DD_\eps} \ti{s}(\zeta) \tilde{\mathcal{G}}(0,\zeta)  d \zeta \right) \ed z = i\pi \ti{s}(0)\ed z = i\pi s(z_0)  
\end{eqnarray*}
where we have used \eqref{eq:localG} to obtain the final equality.  

\end{proof}

\appendix
\section{Palais-Smale sequences for the $H$-energy} 
\label{sec:PS}
 
Here we state a slight modification of a theorem of Brezis-Coron \cite[Theorem 1]{BC85} concerning $\dot H^1$-bubble convergence of Palais-Smale sequences of the $H$-functional $E$. Recall that a Palais-Smale sequence $\{u_k\}\In \dot H^1$ is a sequence with uniformly bounded energy for which $\ed E(u_k)\to 0$ in $H^{-1}$. We note that any Palais-Smale sequence is uniformly bounded in $\dot H^1(\Sigma,\R^3)$ since
\beq\label{eq:E16}
E(u) = \frac16 \|u\|_{\dot H^1}^2 + \frac13\ed E(u)[u]
\eeq
holds for all $u$. We first recall from \cite[Lemma A.1]{BC85} that, up to constants, $\om \in L^1_{loc}(\R^2, \R^3)$ is a solution of
\begin{equation}\label{eq:om}
\text{$- \Dl \om = \D \om \dot\wedge \Db \om$,\quad on $\R^2$ with \quad $0<\int_{\R^2} |\D \om |^2 < \infty$}
\end{equation}
if and only if $\om(z) = \bl \left(\frac{P(z)}{Q(z)}\right)$ for some non-constant co-prime complex polynomials $P,Q$ and where $\bl$ is as defined in \eqref{eq:bubblePi}. Note that $\int_{\R^2} |\D \om |^2= 8\pi \max\{\deg(P),\deg(Q)\}\geq 8\pi$.

\begin{theorem}[cf \cite{BC85}, Theorem 1]\label{thm:PS}
Let $(\Sigma, g)$ be a closed oriented Riemannian surface with a metric of constant curvature. Suppose that $\{u_k\}\In \dot H^1(\Sigma,\R^3)$ is a Palais-Smale sequence for $E$. 
Then a subsequence converges weakly in $\dot H^1(\Sigma,\R^3)$ to some $u$ solving \eqref{eq:Heq} and there exist $L\in \N$,  $a_k^l \to a^l\in \Sigma$, $r_k^l\to 0$, associated solutions $\om^l$ to \eqref{eq:om} and choices of isometry $F_k^l$ as in Remark \ref{rmk:def-Fa} for $l=1,\dots,L, k\in \N$, with 
$$\text{$\left\| u_k -  u - \sum_{l=1}^L  \phi^l_k(\cdot) \left( \om^l((r_k^l)^{-1}F_k^l(\cdot)) - 
\om^l(\infty) \right) \right\|_{H^1(\Sigma)} \to 0$,}$$
where $\phi^l_k \in C^\infty(\Sigma)$ is defined by $\phi^l_k (p) =\phi(F^l_k(p))$ on $\spt(\phi^l_k)\In B_\iota (a^l_k)$ for some fixed $\phi\in C_c^\infty(\mathbb{D}_{\rho},[0,1])$ with $\phi\equiv 1$ on $\mathbb{D}_{\rho/2}$ and $\iota$, $\rho$ are defined as in Remark \ref{rmk:def-Fa}.

\end{theorem}
\begin{remark}
Notice that we cannot easily conclude that any ``stronger'' notions of bubble-tree convergence hold (e.g. convergence in $L^\infty$) unless we assume $f_k:=-\Dl u_k - \D u_k\dot\wedge \Db u_k\to 0$ in a stronger sense (e.g. in $L^2$).
\end{remark}
\begin{remark}\label{rmk:A2}
In contrast to Section \ref{sec:main-proofs} we can use a very rough notion of a glued in bubble on $\Sigma$, since we are not claiming any quantitative asymptotics at this stage. However, as recalled in the introduction, if $\Sigma \not\cong \Sp^2$ then there do not exist critical points $u\in \dot H^1(\Sigma,\R^3)$ of $E$  with $E(u) \leq \frac{4\pi}{3}$. Thus given a Palais-Smale sequence $\{u_k\}$ with $E(u_k)\to \frac{4\pi}{3}$, then the above result yields $\dist_{\dot{H}^1}(u_k ,\mathcal{Z})\to 0$. Thus we may apply our main theorems to obtain quantitative control on convergence properties of such Palais-Smale sequences.  \end{remark}

\begin{proof}[Sketch of the Proof of Theorem  \ref{thm:PS}] We will simply outline how to adapt the proof of \cite[Theorem 1]{BC85} to this setting. The proof there follows from a series of Lemmas 1-4 , two of which we will re-write in our setting and leave the details to the interested reader (each proof follows, more or less, exactly along the lines of its associated proof in \cite{BC85}). Notice first of all that we do not need to re-state Lemmas 2 or 3 in \cite{BC85} since they are already in the form that we require. 

First of all  we may assume without loss of generality that $u_k$ is uniformly bounded in $L^\infty$: let $g_k\in \dot H^1$ solve $\lan g_k, v\ra_{\dot H^1} = \ed E(u_k)[v]$ for all $v\in \dot H^1$, so that $\|g_k\|_{\dot H^1}\to 0$. We set $v_k = u_k - g_k$ and we have  
\begin{equation*}
	-\Dl v_k 
	= \D v_k \dot \wedge \Db v_k + 2 \D v_k \dot \wedge \Db g_k + \D g_k \dot\wedge \Db g_k . 
\end{equation*}
Setting $f_k = 2 \D v_k \dot \wedge \Db g_k + \D g_k \dot\wedge \Db g_k$, estimate \eqref{est:Wente} immediately gives that $\|f_k\|_{H^{-1}}\to 0$ so that $\{v_k\}$ is Palais-Smale and by definition $\|v_k-u_k\|_{\dot H^1}\to 0$. Furthermore, Wente's estimate \eqref{eq:West} gives $\|v_k\|_{L^\infty} \leq C$ for some uniform constant $C$, giving the claimed bounds. We will abuse notation and now continue to work with $u_k$ under the assumption that it is also uniformly bounded in $L^\infty$. 

We trivially have the existence of some $u\in \dot H^1(\Sigma,\R^3)$ so that, for a subsequence, $u_k\rightharpoonup u$ in $H^1$. The first Lemma states that if $u_k$ does not converge strongly to $u$ then we must be able to find a non-trivial bubble by suitably re-scaling the sequence around a shrinking ball: 

\begin{lemma}[cf \cite{BC85} Lemma 1] \label{lem:BC1}
Suppose that $\{u_k\}$ satisfies the hypotheses of Theorem \ref{thm:PS} with $\|u_k\|_{L^\infty}$ uniformly bounded. Then either a subsequence $u_k\to u$ converges strongly in $H^1$, or there exist $a_k\to a\in \Sigma$ and $r_k\to 0$, an associated non-trivial solution $\om$ of \eqref{eq:om} and isometries $F_k$ so that a subsequence, $\tilde{u}_k (\cdot) = u_k(r_k F_k^{-1}(\cdot))\to \om$ almost everywhere in $\R^2$ and $\D \tilde{u}_k\rightharpoonup \D \om$ weakly in $L_{loc}^2(\R^2)$. 
 	\end{lemma}

The final lemma now guarantees that we can remove the resulting bubble $\om$ from $u_k$ and remain a Palais-Smale sequence. 
 
\begin{lemma}[cf \cite{BC85} Lemma 4]\label{lem:BC2}
Suppose that $\{u_k\}$ satisfies the hypotheses of Lemma \ref{lem:BC1} and that $u_k$ does not converge strongly to $u$ in $H^1$. Then with the notation as in Lemma \ref{lem:BC1} and Theorem \ref{thm:PS}
$$\text{$v_k = u_k - \phi_k(\cdot)(\om(r_k^{-1}F_k(\cdot)) - 
\om(\infty))$}$$ is a Palais-Smale sequence which is bounded in $L^\infty$ and $v_k \rightharpoonup u$ in $\dot H^1$. Furthermore 
$$\|\D v_k\|_{L^2(\Sigma)}^2 = \|\D u_k\|_{L^2(\Sigma)}^2 - \|\D \om\|_{L^2(\R^2)}^2 + o(1).$$
\end{lemma}

Repeated applications of the above lemmas is sufficient to conclude the proof of the theorem: if $v_k$ converges strongly to $u$ in $H^1$ then we are done. If not we continue to apply the above lemmas, with each iteration lowering the Dirichlet energy by at least $8\pi$ until we obtain a strongly converging subsequence. 

\end{proof}
\bibliographystyle{plain}

\vspace{.5cm}

\begin{flushleft}
A. Malchiodi: Scuola Normale Superiore, Piazza dei Cavalieri 7, 56126 Pisa, Italy \\\textit{andrea.malchiodi@sns.it}

  M. Rupflin: Mathematical Institute, University of Oxford, Oxford OX2 6GG, UK \\   \textit{melanie.rupflin@maths.ox.ac.uk}

  B. Sharp: School of Mathematics, University of Leeds, Leeds LS2 9JT, UK \\\textit{b.g.sharp@leeds.ac.uk} 
\end{flushleft}

\end{document}